\documentclass[a4paper,10pt]{article}
\usepackage[utf8]{inputenc}

\usepackage{comment,xcomment,amssymb,amsmath,color}
\usepackage{rotating,xypic,array}
\usepackage{latexsym,mathtools,amsthm}
\usepackage{subfig}
\usepackage{tikz,fancyvrb}
\usepackage{booktabs}
\usetikzlibrary{arrows.meta}
\usepackage{enumitem}
\usepackage{url} 
\usepackage{longtable}
\usepackage[section]{placeins}
\usepackage{multirow}

\makeatletter
\def\namedlabel#1#2{\begingroup
   \def\@currentlabel{#2}%
   \label{#1}\endgroup
}
\makeatother
\usepackage[colorlinks, citecolor=blue]{hyperref}
\title{Indefinite nilsolitons and Einstein solvmanifolds}
\author{Diego Conti and Federico A. Rossi}

\newtheorem{theorem}{Theorem}[section]
\newtheorem{lemma}[theorem]{Lemma}

\newtheorem{corollary}[theorem]{Corollary}
\newtheorem{proposition}[theorem]{Proposition}
\theoremstyle{definition}
\newtheorem{definition}[theorem]{Definition}
\newtheorem{example}[theorem]{Example}
\theoremstyle{remark}
\newtheorem{remark}[theorem]{Remark}

\newcommand{\R}{\mathbb{R}}
\newcommand{\lie}[1]{\mathfrak{#1}}     
\newcommand{\g}{\lie{g}}

\newcommand{\h}{\mathbb{H}}

\newcommand{\hook}{\lrcorner\,}

\newcommand{\id}{\mathrm{Id}}   

\newcommand{\gl}{\lie{gl}}

\newcommand{\Span}[1]{\operatorname{Span}\left\{#1\right\}}

\newcommand{\st}{\;\mid\;}          
\DeclareMathOperator{\ric}{ric} 
\DeclareMathOperator{\Ric}{Ric} 

\DeclareMathOperator{\Aut}{Aut}

\DeclareMathOperator{\diag}{diag}
\DeclareMathOperator{\Der}{Der}
\DeclareMathOperator{\ad}{ad}
\DeclareMathOperator{\adtilde}{\widetilde{ad}}

\DeclareMathOperator{\Tr}{tr}

\newcolumntype{C}{>{$}c<{$}}
\newcolumntype{L}{>{$}l<{$}}
\newcolumntype{R}{>{$}r<{$}}

\newcommand{\llangle}{\langle\!\langle}
\newcommand{\rrangle}{\rangle\!\rangle}

\begin{document}
\maketitle

\begin{abstract}
A nilsoliton is a nilpotent Lie algebra $\g$ with a metric such that $\Ric=\lambda \id+D$, with $D$ a derivation. For indefinite metrics, this determines four different geometries,
according to whether $\lambda$ and $D$ are zero or not. We illustrate with examples the greater flexibility of the indefinite case compared to the Riemannian setting.  We determine the algebraic properties that  $D$ must satisfy when it is nonzero.

For each of the four geometries, we show that under suitable assumptions it is possible to extend the nilsoliton metric to an Einstein solvmanifold of the form $\g\rtimes \R^k$. Conversely, we introduce a large class of indefinite Einstein solvmanifolds of the form $\g\rtimes \R^k$ that determine a nilsoliton metric on $\g$ by restriction. We show with examples that, unlike in the Riemannian case, one cannot establish a  correspondence between the full classes of Einstein solvmanifolds and nilsolitons.
\end{abstract}

\renewcommand{\thefootnote}{\fnsymbol{footnote}}
\footnotetext{\emph{MSC class 2020}: \emph{Primary} 53C50; \emph{Secondary} 53C25, 53C30, 22E25}
\renewcommand{\thefootnote}{\arabic{footnote}}

\section*{Introduction}
A Riemannian solvmanifold can be defined as a solvable Lie group endowed with a left-invariant Riemannian metric. A long-standing conjecture of Alekseevski states that all Einstein homogeneous Riemannian manifolds of negative curvature are of this type~\cite{Alekseevski:Classification}; a proof of this conjecture appeared recently in~\cite{BohmLafuente:AlekseevskyConjecture}. The geometry of a Riemannian solvmanifold is entirely determined by assigning the metric at the identity, i.e. by fixing an inner product on the Lie algebra. One then says that the Lie algebra is Einstein if so is the corresponding left-invariant metric.

The structure of Riemannian Einstein solvable Lie algebras is well understood (see the recent survey~\cite{Jablonski:Survey}). In particular, they are nonunimodular~(\cite{Dotti:RicciCurvature}) and \emph{standard}~(\cite{Lauret:Einstein_solvmanifolds}), i.e. they decompose as an orthogonal direct sum
\[\tilde \g=\g\oplus^\perp\lie a,\]
where $\g$ is a nilpotent ideal and $\lie a$ an abelian subalgebra. Previous results of J.~Heber~(\cite{Heber:noncompact}) show that $\lie a$ acts by normal derivations on $\g$; more precisely, for each $X$ in $\lie a$, one has that the metric adjoint of $\ad X$ is a derivation commuting with $\ad \lie a$. A result of R.~Azencott and E.N.~Wilson~(\cite{AzencottWilson2}) then shows that the Lie algebra can be modified by projecting on the self-adjoint part, giving rise to an isometric solvmanifold for which the abelian subalgebra $\lie a$ acts by self-adjoint derivations; it is shown in~\cite{Heber:noncompact} that $\lie a$ acts faithfully and contains some $H$ such that $\ad H$ is positive definite. One then says that $\g$ is of \emph{Iwasawa type}.

The restriction of the metric to the nilpotent ideal $\g$ satisfies the  equation
\begin{equation}
\label{eqn:nilsoliton}
\Ric=\lambda \id+D, \quad \lambda\in\R,\ D\in\Der\g.
\end{equation}
Left-invariant metrics on a Lie group satisfying~\eqref{eqn:nilsoliton} are called \emph{algebraic Ricci solitons}, due to the fact that they are homothetic solitons for the Ricci flow~(\cite{Lauret:RicciSoliton}); an algebraic Ricci soliton on a nilpotent Lie group is called a \emph{nilsoliton}. In fact, every left-invariant Riemannian Ricci soliton metric on a nilpotent Lie group is a nilsoliton~\cite{Lauret:RicciSoliton}; it was proved in \cite{Jablonski1,Jablonski2} that every homogeneous Riemannian Ricci soliton is isometric to an algebraic Ricci soliton. A nilpotent Lie algebra admits at most one Riemannian nilsoliton metric up to isomorphisms and rescaling~(\cite{Lauret:RicciSoliton}). Several characterizations of Lie algebras admitting a nilsoliton metric are known (\cite{Payne:TheExistence,Nikolayevsky,Lauret:RicciSoliton}); this has led to classifications in small dimension (\cite{Will:RankOne, FernandezCulma, Lauret:Finding}).

Conversely, it was shown in~\cite{Lauret:RicciSoliton} that every Riemannian nilsoliton $\g$ gives rise to an Einstein solvable Lie algebra  $\tilde \g=\g\rtimes\R$. This effectively implies that a classification of Einstein Riemannian solvmanifolds can be reduced to a classification of Riemannian nilsolitons.

\smallskip
The notions of algebraic Ricci soliton and nilsoliton carry over naturally to the pseudo-Riemannian setting, simply by imposing~\eqref{eqn:nilsoliton} on an indefinite metric. It is known (see~\cite{Onda:ExampleAgebraicRicciSolitons}) that pseudo-Riemannian algebraic Ricci solitons are Ricci solitons; however, a left-invariant indefinite metric on a Lie group can be a Ricci soliton without satisfying~\eqref{eqn:nilsoliton}, see \cite{BatatOnda:AlgebraicRicciSoliton, Calvaruso:OnSemidirect, Wears}. The variational nature of nilsoliton metrics as critical points of the scalar curvature for an appropriately restricted class of metrics also carries over to the indefinite case (see~\cite{Yan:pseudoalgebraic}).

Moreover, every pseudo-Riemannian nilsoliton with $\lambda\neq0$ and $\Tr D\neq0$ determines an Einstein solvable Lie algebra
\cite{Yan:Pseudo-RiemannianEinsteinhomogeneous}. One known construction of pseudo-Riemannian nilsolitons is a form of Wick rotation (see~\cite{HELLELAND2020:WickRotations})
where the metric of a Riemannian nilsolitons is altered by inverting the sign of the metric on the odd eigenspaces of the  derivation $D$ (suitably normalized; see~\cite{Yan:Pseudo-RiemannianEinsteinhomogeneous}). The double extension procedure of~\cite{MedinaRevoy} can also  be adapted to the nilsoliton setting, yielding a recipe to produce Lorentzian nilsolitons from Riemannian ones~\cite{YanDeng:DoubleExtensionRiemNilsolitons}. Examples of nilsolitons have been constructed in~\cite{Onda:ExampleAgebraicRicciSolitons, KondoTamaru:LorentzianNilpotent}.

A complete understanding of the relation between indefinite nilsolitons and Einstein solvmanifolds appears to be lacking at the time of writing. This paper takes a step in that direction; we construct several examples which show essential differences with the Riemannian case and suggest the appropriate definitions which enable us to extend some known results to the indefinite setting.

We start with a discussion of the correct pseudo-Riemannian analogue of the standard condition. Indeed, the Riemannian definition has several possible formulations, which become equivalent if one assumes the metric to be Einstein.
We show with examples that the equivalence does not hold in the pseudo-Riemannian setting, suggesting that the appropriate definition for the indefinite case is requiring an orthogonal direct sum of a nilpotent ideal and an abelian subalgebra, without insisting that the nilpotent ideal coincide with either the nilradical or the derived algebra. However, we show by an example that not all indefinite Einstein solvmanifolds are standard in this sense.

Another striking difference with the Riemannian case is that in~\eqref{eqn:nilsoliton}, $D$ and $\lambda$ can be independently zero or nonzero, giving a combination of four different geometries.
In particular, we produce an example with $\lambda=0$ and $D\neq 0$; together with the known existing examples, this shows that each of the four groups is nonempty, whilst Riemannian nilsolitons are either flat or satisfy $D\neq 0$, $\lambda<0$. In addition, we observe that the derivation $D$, whilst self-adjoint, is not necessarily semisimple; this can happen for both the case $\lambda=0$ (Example~\ref{ex:Nil2dim4}) and the case $\lambda\neq0$ (Example~\ref{ex:nil4notss}). We also derive algebraic conditions on the derivation $D$ (Theorem~\ref{thm:nilsolitonsandnik}) that allow us to conclude that most Lie algebras do not admit a nonsemisimple nilsoliton metric with $\lambda\neq0$, at least in low dimensions (Proposition~\ref{prop:nonil4}).

The Iwasawa condition also has a natural equivalent for indefinite metrics, which we call \emph{pseudo-Iwasawa}. We show that the Azencott-Wilson theorem extends to the indefinite setting (Proposition~\ref{prop:pseudoAzencottWilson}), and therefore enables one to obtain a pseudo-Iwasawa solvmanifold from a standard solvmanifold such that for all $X\in\lie a$, $(\ad X)^*$ is a derivation commuting with $\lie a$. Unfortunately, the latter condition does not hold for every indefinite Einstein solvmanifolds, as we show with examples. Therefore, not all standard Einstein solvmanifolds can be reduced to the pseudo-Iwasawa case, unlike in the Riemannian setting.

Nevertheless, the pseudo-Iwasawa class shows its importance in the fact that the nilpotent ideal $\g$ is always a nilsoliton (Theorem~\ref{thm:solvablenilsolitoncorrespondence}). Conversely, the Einstein extensions of a nilsoliton constructed in \cite{YanDeng:DoubleExtensionRiemNilsolitons,Yan:Pseudo-RiemannianEinsteinhomogeneous} are pseudo-Iwasawa. More precisely, we obtain a correspondence between Einstein solvable Lie algebras with a pseudo-Iwasawa decomposition and a suitable class of nilsolitons
which resembles closely the Riemannian situation, but complicated by the fact that four different geometries occur:
\begin{itemize}[leftmargin=\parindent]
\item A nonunimodular pseudo-Iwasawa solvable Lie algebra with an Einstein metric of nonzero scalar curvature
determines a nilsoliton metric on the nilpotent ideal $\g$ (Corollary~\ref{cor:Hzero}), with $\lambda\neq0$ and $\Tr D\neq0$; in this case $\g$ coincides with the nilradical (Corollary~\ref{cor:nilradical}). Conversely, a nilsoliton with $\lambda\neq0$ and $D$ not nilpotent (or equivalently $\Tr D\neq0$) can be extended to a solvable nonunimodular Einstein Lie algebra $\g\rtimes\lie a$, where $\lie a$ is any subalgebra of self-adjoint derivations containing $D$ such that the bilinear form
\begin{equation*}
\langle X,Y\rangle_{\Tr{}}=\Tr(XY)
\end{equation*}
is nondegenerate (Theorem~\ref{thm:constructionstandardextension}). The correspondence is one-to-one if the metrics are fixed. Unlike the Riemannian case, a Lie algebra may admit more than one nilsoliton metric with $\lambda\neq0$ and $\Tr D\neq0$; in all the cases we know, the resulting solvable extensions are isomorphic as Lie algebras. We do not know if this is a general fact.

\item A unimodular pseudo-Iwasawa solvable Lie algebra with an Einstein metric of nonzero scalar curvature
determines an Einstein metric on the nilpotent ideal $\g$; also in this case $\g$ coincides with the nilradical. Conversely, an Einstein nilpotent Lie algebra with nonzero scalar curvature can be extended to a unimodular pseudo-Iwasawa solvable Lie algebra $\tilde\g=\g\rtimes\lie a$ for any subalgebra $\lie a$ of self-adjoint derivations for which $\langle,\rangle_{\Tr{}}$ is nondegenerate.

\item A nonunimodular pseudo-Iwasawa solvable Lie algebra with a Ricci-flat metric $\tilde\g$ gives rise to a nilsoliton metric on $\g$ with $\Ric=D$, where $D$ can be zero or not (Corollary~\ref{cor:restrictricciflat}). Conversely, a nilsoliton $\g$ with $\Ric=D$ can be extended to a nonunimodular Ricci-flat solvmanifold if $\Der\g$ contains  a subalgebra $\lie a$ of self-adjoint derivations, not all trace-free, on which $\langle,\rangle_{\Tr{}}$ is zero (Proposition~\ref{prop:extendnil1}, Proposition~\ref{prop:extendnil2}). Note that not every nilsoliton with $\Ric=D\neq 0$ admits such an $\lie a$ (Example~\ref{example:nil2nonsiestende}).

\item A unimodular pseudo-Iwasawa solvable Lie algebra with a Ricci-flat metric $\tilde\g$ determines a Ricci-flat metric on the nilpotent ideal $\g$. Conversely, any Ricci-flat nilpotent Lie algebra $\g$ can be extended to a unimodular Ricci-flat solvmanifold; one can of course take the product $\g\times\R^k$, but more generally one can construct pseudo-Iwasawa Ricci-flat solvmanifolds $\g\rtimes\lie a$,  where $\lie a$ is any subalgebra of $\Der\g$ of trace-free self-adjoint derivations such that $\langle,\rangle_{\Tr{}}$ is zero (Proposition~\ref{prop:extendnil1}).
\end{itemize}

Finally, we show that the Azencott-Wilson trick allows one to extend nilsolitons to Einstein solvmanifolds which are not of pseudo-Iwasawa type (Proposition~\ref{prop:costruiscinonpseudoiwasawa}).

\medskip
\noindent \textbf{Acknowledgments:} The authors acknowledge GNSAGA of INdAM. F.A.~Rossi also acknowledges the Young Talents Award of Universit\`{a} degli Studi di Milano-Bicocca joint with Accademia Nazionale dei Lincei.

\section{Standard Lie algebras and the Ricci tensor}
\label{sec:standard}
In the Riemannian context, an important class of metric Lie algebras consists of standard solvable Lie algebras, that is, solvable Lie algebras $\tilde\g$ with a fixed metric such that $[\tilde \g,\tilde \g]^\perp$ is abelian. In fact, J.~Lauret~\cite{Lauret:Einstein_solvmanifolds} proved that all Riemannian Einstein solvmanifolds are standard, and the structure of standard Riemannian Einstein solvmanifolds had been previously described by J.~Heber in \cite{Heber:noncompact}.

Standard Riemannian Lie algebras decompose as
\begin{equation}
\label{eqn:standard}
\tilde \g=\g\oplus^\perp\lie{a},
\end{equation}
with $\g= [\tilde \g,\tilde \g]$ nilpotent and $\lie{a}$ abelian; the notation $\oplus^\perp$ represents an orthogonal direct sum of vector spaces, and $\lie{a}$ acts nontrivially except when $\tilde \g=\lie{a}$.

By contrast, in the indefinite case, imposing that $[\tilde \g,\tilde \g]^\perp$ is abelian does not imply a decomposition of the form \eqref{eqn:standard}.

\begin{example}
\label{ex:adegenerate}
Take the solvable Lie algebra $\tilde\g$ defined by
\[(e^{14},-e^{24},-2e^{12},0);\]
here and throughout the paper, we will use the language introduced in~\cite{Salamon:ComplexStructures}, and describe Lie algebras by giving the action of the Chevalley-Eilenberg operator $d$ on the dual, which is equivalent to giving the expression for the Lie bracket. Thus, the notation above means that $\g^*$ has a fixed basis $\{e^1,\dotsc, e^4\}$ with $de^1=e^{14}=e^1\wedge e^4$, $de^2=-e^{24}=-e^2\wedge e^4$ and so on.

Consider the metric
\[g_1e^1\otimes e^1+ g_2e^2\otimes e^2+g_3e^3\odot e^4,\]
where $e^3\odot e^4=e^3\otimes e^4+e^4\otimes e^3$ and $g_1,g_2,g_3$ are nonzero real parameters. Computations show that this metric is Einstein when $g_2=g_3^2/g_1$.

In this case, $[\tilde\g,\tilde\g]^\perp=\Span{e_3}$ is one-dimensional, hence abelian; however it is not possible to write the decomposition~\eqref{eqn:standard}
because the metric restricted to $[\tilde\g,\tilde\g]=\Span{e_1,e_2,e_3}$ and  $[\tilde\g,\tilde\g]^\perp$ is degenerate.
\end{example}
The above example shows that in order for an Einstein pseudo-Riemannian solvable Lie algebra to be decomposed as the orthogonal direct sum of a nilpotent ideal and an abelian subalgebra, one needs to impose that the restriction of the metric to (one of) the factors is nondegenerate.

In addition, for Riemannian Einstein solvable Lie algebras one has that $[\tilde \g,\tilde \g]$ equals the nilradical of $\tilde \g$ (see \cite{Heber:noncompact}), but this is not always true in the pseudo-Riemannian setting. This means that there is some freedom in the choice of the nilpotent ideal $\g$. Insisting that $\g$ equal either the derived algebra or the nilradical is not appropriate, as shown by the following examples.
\begin{example}\label{ex:derivedalgebradoesnotwork}
Consider the nilpotent Lie algebra $\g$
\[(0, 0, 0, e^{12}, e^{13},e^{24}, e^{15} + e^{23},e^{26} + e^{14})\]
and take the semidirect product
\[\tilde\g=\g\oplus\Span{e_{9}}, \quad \ad e_{9}=e^3\otimes e_3 + e^5\otimes e_5 + e^7\otimes e_7.\]
For any metric of the form  $\sum g_i e^i\otimes e^i$, this gives an orthogonal decomposition of $\tilde \g$ as a direct sum of a nilpotent ideal (the nilradical in this case) and an abelian subalgebra, in analogy to \eqref{eqn:standard}. However, the nilpotent ideal cannot be taken to be $[\tilde \g,\tilde \g]$, whose orthogonal complement $\Span{e_1,e_2,e_{9}}$ is not a subalgebra.  We point out that the $g_i$ can be chosen so that the  metric is Einstein (of indefinite signature); we refer to Example~\ref{ex:extensionNONpseudo-Iwasawa} for the calculations.
\end{example}

\begin{example}
Consider the solvable Lie algebra $(0,e^{12},-e^{13},0)$ with invariant Ricci-flat metric given by:
\[\langle,\rangle=e^1\odot e^4+e^2\odot e^3.\]
It is easy to verify that the metric restricted to the nilradical $\Span{e_2,e_3,e_4}$ is degenerate. Instead there is an orthogonal decomposition as a direct sum of the nilpotent ideal $\g=\Span{e_2,e_3}=[\tilde{\g},\tilde{\g}]$ and the abelian Lie algebra $\lie{a}=\Span{e_1,e_4}$.
\end{example}

Motivated by the above examples, in the attempt to generalize Heber's results, we shall employ the following:
\begin{definition}
\label{def:standard}
A \emph{standard decomposition} of a metric Lie algebra $\tilde \g$ is a decomposition
\[\tilde \g=\g\oplus^\perp\lie a,\]
where $\g$ is a nilpotent ideal and $\lie a$ is an abelian subalgebra.
\end{definition}
Thus, a standard Riemannian solvable Lie algebra $\tilde\g$ admits the standard decomposition~\eqref{eqn:standard}. However, a Riemannian solvable Lie algebra may admit a standard decomposition even if it not standard, as evident from Example~\ref{ex:derivedalgebradoesnotwork} (although a nonstandard Riemannian Lie algebra cannot be Einstein \cite{Lauret:Einstein_solvmanifolds}).

\begin{remark}
\label{rk:g_has_to_be_derived_algebra}
If $\tilde \g$ admits a standard decomposition $\tilde \g=\g\oplus^\perp\lie a$, then $\g$ sits between the derived algebra and the nilradical of $\tilde \g$. Thus, if the nilradical coincides with $[\tilde \g,\tilde \g]$ (as is the case for Riemannian Einstein solvable Lie algebras, see~\cite{Heber:noncompact}), the only possible standard decomposition is with $\g=[\tilde \g,\tilde \g]$.
\end{remark}

In the rest of this section we shall study the structure of Einstein pseudo-Riemannian Lie algebras with a standard decomposition. We emphasize, however, that not all Einstein pseudo-Riemannian solvmanifolds are of this type.
\begin{example}
The metric Lie algebra of Example~\ref{ex:adegenerate} does not admit a standard decomposition $\g\oplus^\perp\lie a$: by Remark~\ref{rk:g_has_to_be_derived_algebra}, in this case the only possible choice for $\g$ is the derived algebra $\Span{e_1,e_2,e_3}$, which is degenerate.
\end{example}
\begin{example}
Even when the derived algebra is nondegenerate,  standard decompositions may fail to exist.

Consider the solvable Lie algebra $\tilde\g$ defined by
\[ [e_1,e_2]=e_1,\quad [e_1,e_3]=e_1,\quad [e_1,e_4]=e_1,\quad [e_3,e_4]=-2(e_2-e_3);\] this is isomorphic to $\lie{aff}_{\R}\times\lie{aff}_{\R}=(0,e^{12},0,e^{34})$, as one can verify by considering the basis  $\{e_1,e_2,e_2-e_3,\frac12(e_4-e_2)\}$.
Fix the metric
\[g=e^1\otimes e^1+ e^2\otimes e^2 + e^3\odot e^4;\]
computations show that this metric is Einstein with $\lambda=-3$.

In this case, the nilradical coincides with the derived algebra; we compute
\[[\tilde\g,\tilde\g]=\Span{e_1,e_2-e_3}, \qquad[\tilde\g,\tilde\g]^\perp=\Span{e_3,e_2+e_4}.\]
Thus, whilst nondegenerate, the orthogonal complement of $[\tilde\g,\tilde\g]^\perp$ is not abelian, and does not give a standard decomposition, since $[e_3,e_2+e_4]=-2(e_2-e_3)$.

Notice however, that we can write a \emph{nonorthogonal} direct sum
\[\tilde\g=[\tilde\g,\tilde\g]\oplus \Span{e_2,e_4},\]
with $\Span{e_2,e_4}$ abelian.
\end{example}

For the rest of this section, $\tilde\g$ will denote a pseudo-Riemannian solvable Lie algebra with a standard decomposition $\tilde\g=\g\oplus^\perp\lie{a}$.
We can choose an orthogonal basis $\{e_1,\dotsc, e_k\}$ of $\lie{a}$, $\langle e_k,e_k\rangle=\epsilon_k=\pm1$. Let $\phi_1,\dots, \phi_k$ be the derivations of $\lie{g}$ defined by
\[[v,e_s]=\phi_s(v)\quad v\in\g,\]
and let $\phi^*_s$ denote the adjoint with respect to the pseudo-Riemannian metric $\langle,\rangle$, i.e. $\langle X,\phi_s(Y)\rangle=\langle \phi_s^*(X),Y\rangle$ for all $X,Y\in\tilde{\g}$.

With the above notation, we have the following lemmas.

\begin{lemma}\label{lemma:structuresofgandtildeg}
The Lie algebra structures of $\g$ and $\tilde\g$ are related by
\[\adtilde v = \ad v +\sum_s e^s\otimes \phi_s(v), \qquad \adtilde e_s=-\phi_s,\]
and as a consequence
\[\Tr \adtilde v = 0, \qquad \Tr \adtilde e_s=-\Tr\phi_s.\]
For the exterior covariant derivative, we obtain
\[\tilde dv^\flat = dv^\flat -\sum_s(\phi_s^*(v))^\flat\wedge e^s, \qquad \tilde de^s=0.\]
\end{lemma}
\begin{proof}
The first part is obvious; the second follows from
\[\tilde dv^\flat(u,e_s)=-\langle v,[u,e_s]\rangle = -\langle v,\phi_s(u)\rangle = - \langle \phi_s^*(v),u\rangle.\qedhere\]
\end{proof}

\begin{lemma}\label{lemma:killingofgandtildeg}
The Killing form of $\tilde\g$ satisfies
\[B(v,w)=0, \qquad B(v,e_s)=0, \qquad B(e_s,e_r)=\Tr(\phi_s\circ\phi_r),\]
for any $v,w\in\g$, $e_s,e_r\in\lie{a}$.
\end{lemma}
\begin{proof}
By \cite[Chapter~I, Section~5.5]{Bourbaki:LieGropuCh123}, the kernel of $B$ contains every nilpotent ideal of $\tilde{\g}$, so in particular it contains $\g$. The last formula holds by construction.
\end{proof}

Recall the general formula~\cite[Lemma 1.1]{ContiRossi:EinsteinNilpotent} for the Ricci tensor of a metric Lie algebra $\tilde\g$:
\begin{equation}
\label{eqn:ricciingeneral}
\widetilde{\ric}(v,w)=\frac12\langle \tilde dv^\flat, \tilde dw^\flat\rangle -\frac12\langle \adtilde v, \adtilde w\rangle-\frac12 \Tr\adtilde(v\hook \tilde dw^\flat+w\hook \tilde dv^\flat)^\sharp-\frac12 B(v,w).
\end{equation}

\begin{proposition}
\label{prop:ricciofstandard}
The Ricci tensor of the metric $\langle,\rangle$ on $\tilde \g$ and its restriction to $\g$ are related by
\begin{align*}
\widetilde{\ric}(v,w)&=\ric(v,w)+\sum_s\frac12\epsilon_s\langle[\phi_s,\phi_s^*](v), w\rangle  -\frac12\epsilon_s\langle (\phi_s+\phi_s^*)(v),w\rangle \Tr\phi_s\\
\widetilde{\ric}(v,e_s)&=\frac12\langle \ad v, \phi_s\rangle
\end{align*}
and
\[
\widetilde{\ric}(e_s,e_r)=  - \frac12 \langle \phi_s,\phi_r\rangle -\frac12 \Tr(\phi_s\circ\phi_r).
\]
\end{proposition}
\begin{proof}
Applying~\eqref{eqn:ricciingeneral} to  $v,w\in\g$ and using Lemmas~\ref{lemma:structuresofgandtildeg} and~\ref{lemma:killingofgandtildeg}, we compute
\begin{align*}
\widetilde{\ric}(v,w)
&=\frac12\langle dv^\flat,  dw^\flat\rangle +\sum_s\frac12\epsilon_s\langle\phi_s^*(v), \phi_s^*(w)\rangle-\frac12\langle  \ad v,  \ad w\rangle\\
&\quad-\sum_s\frac12\epsilon_s\langle \phi_s(v),\phi_s(w)\rangle
+\frac12\epsilon_s(\langle\phi_s^*(w),v\rangle+\langle\phi_s^*(v),w\rangle)\Tr\adtilde e_s\\
&=\ric(v,w)+\frac12\sum_s\epsilon_s\bigl(\langle\phi_s^*(v), \phi_s^*(w)\rangle  -\langle \phi_s(v),\phi_s(w)\rangle\bigr)\\
&\quad-\frac12\epsilon_s\bigl(\langle \phi_s(v),w\rangle+\langle \phi_s(w),v\rangle\bigr)\Tr\phi_s.
\end{align*}
The others are similar.
\end{proof}
\begin{remark}
Using Proposition~\ref{prop:ricciofstandard}, we deduce a formula relating the Ricci operator $\widetilde{\Ric}$ of $\tilde \g=\g\ltimes_\phi\R e_0$ and the Ricci operator $\Ric$ of $\g$, extended to $\tilde\g$ by declaring it to be zero on the $\R$ factor:
\begin{equation}\label{eqn:riccioperatorstandard}
\widetilde{\Ric}=\Ric+\frac{1}{2}\epsilon_0\bigl([\phi,\phi^*]-(\phi+\phi^*)\Tr\phi +\langle \ad,\phi\rangle e_0-(\langle \phi,\phi\rangle+\Tr(\phi^2))e^0\otimes e_0 \bigr).
\end{equation}
\end{remark}

The formula of Proposition~\ref{prop:ricciofstandard} simplifies if we assume an additional condition, namely that the $\phi_s$ are \emph{normal}, i.e. $[\phi_s,\phi_s^*]=0$. This condition is automatic in the case of Einstein standard Riemannian Lie algebras, due to the following:
\begin{theorem}[{\cite[Theorem 4.10, Lemma 2.1]{Heber:noncompact}}]
\label{thm:normal}
Let $\tilde \g=\g\oplus\lie{a}$ be a Riemannian, standard, Einstein solvable Lie algebra. Then for any $X\in\lie{a}$, $(\ad X)^*$ is also a derivation commuting with $\ad\lie a$.
\end{theorem}
The importance of this result is that by~\cite[Proposition 2.5]{Heber:noncompact} (based on~\cite[Lemma 2.1]{Heber:noncompact} and~\cite[Lemma 4.2]{AzencottWilson2}) the $\phi_s$ can then be assumed to be symmetric, up to considering a different, isometric solvmanifold. In part, these arguments apply to the pseudo-Riemannian case too (see Proposition~\ref{prop:pseudoAzencottWilson} below). However, Theorem~\ref{thm:normal} itself does not extend to indefinite signatures of the metric (nor even Lorentzian), as shown by the following:

\begin{example}
\label{ex:adXnonnormale}
Consider the Lie algebra $\tilde\g=\R^3\rtimes_D\Span{e_4}$, with
nonvanishing brackets  given by
\[[e_1,e_4]=-2e_1,\quad [e_2,e_4]= -5 e_2+ 6e_3,\quad [e_3,e_4]= e_3.\]
The Lorentzian metric given by
\[e^1\otimes e^1 +e^2\otimes e^2-e^3\otimes e^3+e^4\otimes e^4\]
is Einstein with Einstein constant $\lambda =-12$; by construction, we have a standard decomposition $\tilde \g=\Span{e_1,e_2,e_3}\oplus^\perp\Span{e_4}$.

We have
\[\ad e_4=
\begin{pmatrix}
 2 & 0 & 0 & 0 \\
 0 & 5 & 0 & 0 \\
 0 & -6 & -1 & 0 \\
 0 & 0 & 0 & 0 \\
\end{pmatrix},
\quad
(\ad e_4)^*=
\begin{pmatrix}
 2  & 0 & 0 & 0 \\
 0 & 5  & 6  & 0 \\
 0 & 0 & -1  & 0 \\
 0 & 0 & 0 & 0 \\
\end{pmatrix};
\]
thus, $\ad e_4$ is not normal; this implies in particular that $(\ad e_4)^*$ is not a derivation of $\tilde \g$, though it is a derivation of the abelian Lie algebra $\Span{e_1,e_2,e_3}$.
\end{example}

\begin{example}\label{ex:AdjointNotDerivationExample1}
Take $\tilde{\g}=(0,2e^{12},e^{13},3e^{14}+e^{23})$, with the metric
\[g_1e^1\otimes e^1+g_2e^2\otimes e^2+g_3e^3\odot e^4.\]
This is an Einstein metric with $\lambda = -12/g_1$. However we have
\[\ad e_1 = \begin{pmatrix} 0 & 0 & 0 &0\\ 0 & -2 & 0 & 0 \\ 0 & 0 &-1 &0 \\ 0& 0 & 0 &- 3 \end{pmatrix}, \quad
(\ad e_1)^* = \begin{pmatrix} 0 & 0 & 0 &0\\ 0 & -2 & 0 & 0 \\ 0 & 0 & -3 &0 \\ 0& 0 & 0 & -1\end{pmatrix}.\]
Thus, $\ad e_1$ is normal, but $(\ad e_1)^*$ is not a derivation.
\end{example}
In the first Example~\ref{ex:adXnonnormale} the signature of the metric on $\g$ and $\tilde{\g}$ is Lorentzian; in the second Example~\ref{ex:AdjointNotDerivationExample1} it varies according to $g_1,g_2$; this shows that the pseudo-Riemannian version of Theorem~\ref{thm:normal} fails even in the indefinite signature which is ``closest'' to Riemannian.

Nevertheless, if one \emph{imposes} that the derivations $\ad X$ are normal, a useful consequence can be drawn. Let $H$ be the metric dual of $v\mapsto \Tr \adtilde v$, i.e.
\begin{equation}\label{eqn:definitionH}
\langle H,v\rangle = \Tr(\adtilde v),\quad v\in\tilde{\g}.
\end{equation}

\begin{lemma}
\label{lemma:derivationphi}
On a solvable Lie algebra with a standard decomposition,
\begin{equation}
\label{eqn:derivationphi}
H=-\sum_s \epsilon_s (\Tr\phi_s) e_s, \qquad \ad H = \sum_s \epsilon_s (\Tr \phi_s)\phi_s.
\end{equation}
\end{lemma}
\begin{proof}
For $v\in \g$, we have $\langle H,v\rangle =\Tr\ad v=0$, since $\g$ is nilpotent. On the other hand
\[\langle H,e_s\rangle = -\Tr\phi_s = \Tr \ad e_s ,\]
Hence $H=-\sum_s \epsilon_s (\Tr\phi_s) e_s$. The expression of $\ad H$ is then trivial.
\end{proof}

As shown in \cite[Remark 1.3]{ContiRossi:EinsteinNilpotent}, equation~\eqref{eqn:ricciingeneral} can be rewritten as
\begin{equation}
\label{eqn:ricciingeneral2version}
\widetilde{\ric}(v,w)=\frac12\langle \tilde dv^\flat, \tilde dw^\flat\rangle -\frac12\langle \adtilde v, \adtilde w\rangle+\frac12 \bigl(\langle [v,H],w\rangle + \langle [w,H],v\rangle\bigr)-\frac12 B(v,w).
\end{equation}

\begin{proposition}
\label{prop:subalgebraEinstein}
Let $\tilde \g$ be a solvable Lie algebra with an Einstein metric and a standard decomposition $\tilde \g=\g\oplus^\perp\lie a$; assume that $\ad X$ is normal for all $X$ in $\lie a$. If $\lie{b}$ is a subspace of $\lie a$ containing $H$ such that the restriction of the metric is nondegenerate, then the subalgebra $\g\oplus^\perp\lie{b}\subset\g\oplus^\perp\lie{a}$ is also Einstein.
\end{proposition}
\begin{proof}
We can choose the orthonormal basis $\{e_s\}$ so that $\lie b=\Span{e_1,\dotsc, e_k}$; we allow $k=0$ for the case in which both $H$ and $\lie b$ are zero. Since $H\in\lie b$, we have $\Tr\phi_s=-\langle e_s,H\rangle=0$ for $s>k$. Using the fact that the $\phi_s$ are normal, i.e. $[\phi_s,\phi_s^*]=0$,
the formulae of Proposition~\ref{prop:ricciofstandard} simplify to
\begin{gather*}
\widetilde{\ric}(v,w)=\ric(v,w)
-\frac12\sum_r\epsilon_r\bigl(\langle (\phi_r+\phi_r^*)(v),w\rangle)\Tr\phi_r,\\
 \widetilde{\ric}(v,e_s)=\frac12\langle \ad v, \phi_s\rangle
\end{gather*}
and
\[
\widetilde{\ric}(e_s,e_r)=  - \frac12 \langle \phi_s,\phi_r\rangle -\frac12 \Tr(\phi_s\circ\phi_r).
\]
Thus, the Ricci tensor of $\g\oplus^\perp
\lie b$ coincides with the restriction of the Ricci of $\g\oplus^\perp\lie a$, which is a multiple of the metric; hence, $\g\oplus^\perp\lie b$ is Einstein.
\end{proof}
\begin{remark}
In the unimodular case, i.e. for $H=0$, we can choose $\lie{b}=\{0\}$ in Proposition~\ref{prop:subalgebraEinstein}; this proves that when $\tilde\g=\g\oplus^\perp\lie a$ is a unimodular Einstein solvable Lie algebra with a standard decomposition, then $\g$ is also Einstein. We refer to Example~\ref{ex:einsteindentroeinstein} for such an example.
\end{remark}

\begin{remark}
If the Lie algebra is not unimodular, $\langle H,H\rangle$ can be zero or not. If it is zero, the subalgebra $\lie b$ must be taken of dimension at least two (compare with Proposition~\ref{prop:extendnil1}). If if it nonzero, one obtains that $\g\oplus^\perp\Span{H}$ is Einstein, in analogy with the Riemannian case.
\end{remark}

The condition appearing in the statement of Theorem~\ref{thm:normal}, namely that each $(\ad X)^*$ is also a derivation commuting with $\ad \lie{a}$, is strictly stronger than imposing that each $\ad X$ is normal, as can be seen from Example~\ref{ex:AdjointNotDerivationExample1}. If this stronger condition is imposed, it turns out that the $\ad X$ can be assumed to be symmetric.

Indeed, given a metric Lie algebra with a standard decomposition $\tilde \g=\g\oplus^\perp\lie a$, for a given endomorphism $f\colon\g\to\g$, let $f=f^s+f^a$ be the decomposition into a self-adjoint and an anti-selfadjoint part. Define
\[\chi\colon\lie a\to\gl(\g), \quad \chi(X)=(\ad X)^s.\]

In analogy to the Riemannian case (see \cite[Section 1.8]{EberleinHeber}), we have:
\begin{proposition}
\label{prop:pseudoAzencottWilson}
Let $\tilde \g$ be a pseudo-Riemannian Lie algebra with a standard decomposition $\tilde{\g}=\g\oplus^\perp \lie{a}$, and suppose that, for every $X$ in $\lie{a}$, $(\ad X)^*$ is a derivation of $\g$ commuting with $\ad\lie{a}$. Let $\tilde{\g}^*$ be the solvable Lie algebra $\g\rtimes_\chi\lie{a}$. Then there is an isometry between the connected, simply connected Lie groups with Lie algebras $\tilde{\g}$ and $\tilde{\g}^*$, with the corresponding left-invariant metrics, whose differential at $e$ is the identity of $\g\oplus\lie{a}$ as a vector space.
\end{proposition}
\begin{proof}
The proof follows~\cite{EberleinHeber} and \cite[Lemma 4.2]{AzencottWilson2}.

Observe first that for every $X$ in $\lie a$, $\chi(X)=(\ad X)^s$ is a derivation of $\g$ that commutes with $\ad \lie a$, and therefore a derivation of $\tilde \g$. In addition, the image of $\chi$ is an abelian subalgebra and $\chi$ is a homomorphism. Therefore, the semidirect product
$\g\rtimes_\chi\lie a$ is well defined.

The simply connected Lie group $\tilde G$ with Lie algebra $\tilde \g$ has the form $\tilde G=(\exp\g)(\exp \lie a)$. Consider the group
\[H=\Aut(\tilde G)\ltimes \tilde G,\]
with product law
\[(\psi,\tilde{g})(\psi',\tilde{g}')=(\psi \psi',\tilde{g}\psi(\tilde{g}')).\]
The Lie algebra of $H$ is $\lie h=\Der(\tilde \g)\ltimes\tilde \g$, with Lie bracket
\[[(\eta,Y),(\eta',Y')]=([\eta,\eta'],\eta(Y')-\eta'(Y)+[Y, Y']).\]
We let $H$ act on $\tilde G$ by
\[\rho((\psi,\tilde g),\tilde g')=\tilde g\psi (\tilde g').\]
For $X$ in $\lie a$, write $\ad X=(\ad X)^a+(\ad X)^s$, where $\chi(X)=(\ad X)^s$ is the self-adjoint part of $\ad X$; relative to the standard decomposition $\tilde \g=\g\oplus\lie a$, we have
\begin{equation}
\label{eqn:adXaXzero}
 (\ad X)^a=\begin{pmatrix} * & 0 \\ 0 & 0 \end{pmatrix}.
 \end{equation}
Define
\[f\colon\tilde \g\to \Der(\tilde{\g})\times \tilde \g, \quad f(v+ X)=(-(\ad X)^a,v+ X), \;v\in \g, X\in\lie a.\]
The image of $f$ is a subalgebra of $\lie h$ isomorphic to $\g\rtimes_\phi\lie a$: for $v,v'$ in $\g$, $X,X'$ in $\lie a$, we have
\begin{align*}
[f(v),f(v')]&=f([v,v'])\\
[f(X),f(v)]&=[(-(\ad X)^a,X),(0,v)]
=(0,-(\ad X)^a(v)+[X,v])\\
&=(0,(\ad X)^s(v))=f(\chi(X)(v))\\
[f(X),f(X')]&=[(-(\ad X)^a,X),-(\ad X')^a,X')]\\
&=([(\ad X)^a,(\ad X')^a],[X,X'])=0,
\end{align*}
where we have used~\eqref{eqn:adXaXzero} and the fact that
\[0=[\ad X,\ad X']^a=[(\ad X)^a,(\ad X')^a]+[(\ad X)^s,(\ad X')^s].\]

Call $G^*$ the connected subgroup of $H$ with Lie algebra $\tilde \g^*=f(\tilde\g)$. Since $(\ad X)^a$ is anti-selfadjoint and the metric is left-invariant, the action of $G^*$ on $\tilde G$ preserves the metric.

Now observe that for $v$ in $\g$ we have $\rho(\exp f(v),\tilde g)=(\exp v)\tilde g$, and for $X$ in $\lie a$ we get
\begin{multline*}
\rho\bigl(\exp f(X),\tilde g\bigr)=\rho\bigl(\exp (-(\ad X)^a,X),\tilde g\bigr)\\
= \rho\bigl(( \exp (-\ad X)^a,\exp X),\tilde g\bigr)
= \exp X \exp (-\ad X)^a (\tilde g),
\end{multline*}
where we use the fact that $((\ad X)^a,0)$ and $(0,X)$ commute thanks to~\eqref{eqn:adXaXzero}.

We claim that the action of $G^*$ is transitive. Since $\tilde G$ is connected, it suffices to prove that all orbits are open. For fixed $\tilde g\in \tilde G$, we must show the surjectivity of the map
\[d\rho_{e,\tilde g}\colon T_{e}G^*\times \{0\}\to T_{\tilde g}\tilde G.\]
It is clear that
\begin{equation}
 \label{eq:differentialv}
d\rho_{e,\tilde g}(f(v),0)=R_{\tilde g*}v, \quad v\in\g.
\end{equation}
In addition,
\begin{equation}
 \label{eq:differentialX}
d\rho_{e,\tilde g}(f(X),0)=R_{\tilde g*}\left(X+ \frac{d}{dt}|_{t=0}\exp(-\ad tX)^a(\tilde g){\tilde g}^{-1}\right).
\end{equation}
We can write the general element of $\tilde G$ as $\tilde g=gh$, with $g$ in $\exp\g$ and $h$ in $\exp\lie a$; then, using~\eqref{eqn:adXaXzero},
\[
\exp(-\ad tX)^a(\tilde g)\tilde g^{-1}
=\exp(-\ad tX)^a(g) h h^{-1} g^{-1}
=\exp(-\ad tX)^a(g)g^{-1}\in G.
\]
Summing up, $d\rho_{e,\tilde g}(f(X),0)$ is in $R_{\tilde g*}(X+\lie g)$; together with~\eqref{eq:differentialv}, this shows that orbits are open, i.e. the action is homogeneous. Thus, we obtain a covering map $G^*\to \tilde G$; since $\tilde G$ is simply connected, this is a diffeomorphism, inducing a left-invariant metric on $G^*$.

Substituting $\tilde g=e$ in~\eqref{eq:differentialv} and~\eqref{eq:differentialX} shows that the differential at $e$ of the diffeomorphism $G^*\to\tilde G$ is the identity; therefore the pull-back metric on $G^*$ is the same as the left-invariant metric determined by the metric on $\tilde \g^*\cong \tilde \g$.
\end{proof}
Lie algebras of the form $\g\rtimes_\chi\lie a$ as obtained applying Proposition~\ref{prop:pseudoAzencottWilson} will be studied in Section~\ref{sec:correspondence}. Notice however that not all Einstein Lie algebras with a standard decomposition satisfy the hypotheses of Proposition~\ref{prop:pseudoAzencottWilson} (see Examples~\ref{ex:adXnonnormale} and~\ref{ex:AdjointNotDerivationExample1}).

\begin{example}
\label{ex:heisenberg}
Take the Heisenberg Lie algebra $\lie{h}=(0,0,e^{12})$ with a pseudo-Riemannian metric $e^1\otimes e^1+e^2\otimes e^2+g_3e^3\otimes e^3$. Then
\[\Ric=
 \frac{g_3}{2}
 \begin{pmatrix}
-1& 0 & 0\\
0&-1&0\\
0&0&1
\end{pmatrix},\]
and for all $h,k$
\[\phi=
\begin{pmatrix}
1&h+k&0\\h-k&1&0\\0&0&2
\end{pmatrix}
\]
is a derivation, whose symmetric and anti-symmetric parts are
\[\phi^a=
\begin{pmatrix}
0&k&0\\-k&0&0\\0&0&0
\end{pmatrix},
\quad
\phi^s=
\begin{pmatrix}
1&h&0\\h&1&0\\0&0&2
\end{pmatrix}.
\]
For $h=0$, $\phi$ is normal and $\phi^s$ is again a derivation, so $\lie{h}\rtimes_\phi\R$ and $\lie{h}\rtimes_{\phi^s}\R$ are isometric by Proposition~\ref{prop:pseudoAzencottWilson}.

Now assume $h=0$ and, for definiteness, $k=1$. Both $\lie{h}\rtimes_\phi\R$ and $\lie{h}\rtimes_{\phi^s}\R$ carry the Einstein metrics
\[
 e^1\otimes e^1+e^2\otimes e^2\pm (e^3\otimes e^3+\frac14 e^0\otimes e^0).
\]
Recall that a real solvable Lie algebra $\g$ is \emph{completely solvable} (or \emph{split solvable}) if and only if the eigenvalues of all $\ad X$, $X \in \g$, are in $\R$ (see \cite[Corollary 1.30]{Knapp:LieBook}). Hence we observe that $\lie{h}\rtimes_\phi\R$ and $\lie{h}\rtimes_{\phi^s}\R$ are not isomorphic Lie algebras, since only the latter is completely solvable, because the eigenvalues of $\phi^s$ are in $\R$ while the eigenvalues of $\phi$ are complex.
\end{example}

\section{Nilsolitons}
\label{sec:nilsolitons}
A nilpotent Lie algebra $\g$ with a fixed pseudo-Riemannian metric $g$ is a nilsoliton if
\[\Ric=\lambda \id+D, \quad \lambda\in\R,\ D\in\Der\g.\]
This equation forces $D$ to be self-adjoint, but not necessarily semisimple.

Recall that $\gl(n,\R)$ has a natural nondegenerate scalar product defined by
\[\langle X,Y\rangle_{\Tr}=\Tr (X\circ Y);\]
this restricts to a scalar product on $\Der\g$, which may be degenerate.

Recall from~\cite{Nikolayevsky} that a Nikolayevsky derivation is a semisimple derivation $N$ such that
\begin{equation}
 \label{eqn:nik}
\Tr{X}=\Tr{(N\circ X)}, \quad X\in\Der\g.
\end{equation}
It is customary to refer to $N$ as \emph{``the'' Nikolayevsky derivation} (or \emph{``the'' pre-Einstein derivation}), because it is unique up to Lie algebra automorphisms.

\begin{theorem}
\label{thm:nilsolitonsandnik}
Let $g$ be a nilsoliton metric on a nilpotent Lie algebra $\g$. Then either
\begin{enumerate}
\item $\lambda=0$ and $D$ is a nilpotent derivation in the null space of $\Der\g$; or
\item $\lambda\neq0$ and setting $\tilde D=-\frac1\lambda D$, we have
\[\Tr(X)=\Tr(\tilde D\circ X),  \quad X\in\Der\g;\]
relative to the Jordan decomposition $\tilde D=\tilde D_s+\tilde D_n$, $\tilde D_s$ is a Nikolayevsky derivation and $\tilde D_n$ a nilpotent derivation in the null space of $\Der\g$.
\end{enumerate}
In either case, the eigenvalues of $D$ are rational.
\end{theorem}
\begin{proof}
By \cite[Theorem 3.8]{ContiRossi:EinsteinNilpotent}, we have
\[\langle \Ric, X\rangle_{\Tr{}} = \frac14\llangle Xd,d\rrangle,\]
where $\llangle,\rrangle$ is the indefinite scalar product on $\Lambda^2\R^n\otimes\R^n$ induced by $g$.

In particular, for any $X\in\Der\g$ we have $\langle \Ric,X\rangle_{\Tr}=0$. Imposing the nilsoliton condition we find
\[0=\langle \lambda \id+D,X\rangle_{\Tr} = \lambda\Tr(X)+\Tr(D\circ X).\]

So if $\lambda=0$ we see that $D$ is in the null space of $\Der\g$. As observed in \cite[Proof of Theorem 1]{Nikolayevsky}, this space consists of nilpotent derivations: indeed, take the Levi decomposition $\Der \g = \lie{s}\ltimes\lie{r}$, with $\lie s$ semisimple and $\lie r$ the radical. Since $\Der \g$ is algebraic, by \cite[Theorem 4]{Chevalley} we can write $\lie{r}=\lie a\ltimes\lie n$, where $\lie a$ consists of semisimple derivations, $\lie n$ of nilpotent derivations, and $[\lie s,\lie a]=0$; in addition, $\lie n$ is the nilradical of $\Der\g$ and $\lie r$ is algebraic.

Since $\langle,\rangle_{\Tr}$ is ad-invariant, its null space is an ideal; by Cartan's criterion, it is also solvable, so it is contained in $\lie r$.

It is clear that $\langle,\rangle_{\Tr}$ is nondegenerate on $\lie a$.
The restriction to $\lie r\otimes \lie n$ is zero by Lie's theorem. In order to prove that the null space is $\lie n$, it remains to show that the restriction to $\lie s\otimes \lie n$ is also zero. For $k>0$, denote by $V^k$ the space of elements $v$ in $\g$ such that $f_1\dotsm f_k(v)=0$ for $f_1,\dotsc, f_k\in \lie n$, and let $V^0=\{0\}$. Since each $f\in\lie n$ maps $V^k$ to
$V^{k-1}$ and each $g\in\lie s$ preserves the $V^k$, we see that $\langle f,g\rangle_{\Tr{}}=0$.

If $\lambda\neq0$, it is clear that $\tilde D$ satisfies~\eqref{eqn:nik}. Therefore, if $N\in\lie r$ is a Nikolayevsky derivation, $\tilde D -N$ is in the null space of $\Der \g$; in particular, $\tilde D$ is in $\lie r$. Writing the Jordan decomposition $\tilde D=\tilde D_s+\tilde D_n$ in the algebraic Lie algebra $\lie r$, one sees that $\tilde D_n$ must belong to $\lie n$, so $\tilde D_s$ is a semisimple derivation satisfying~\eqref{eqn:nik}. Thus, the eigenvalues of $\tilde D$ are also eigenvalues as $N$, hence rational by~\cite{Nikolayevsky}.
\end{proof}

\begin{corollary}\label{cor:TracciaDquadro}
On a nilsoliton the following equations hold:
\begin{gather*}
\Tr D^2 = -\lambda\Tr D,\\
\Tr\Ric^2=\lambda \Tr\Ric.
\end{gather*}
\end{corollary}
\begin{proof}
The first equation was proved in the proof of Theorem~\ref{thm:nilsolitonsandnik}, and for the second we have:
\[\Tr \Ric^2 =\langle \Ric, \Ric\rangle_{\Tr}=\langle \Ric,\lambda\id+D\rangle_{\Tr}=\lambda \Tr\Ric+\langle\Ric,D\rangle_{\Tr}= \lambda\Tr \Ric.\qedhere\]
\end{proof}

Thus, the indefinite nilsoliton equation corresponds to four different situations:
\begin{enumerate}[label={$(\mathrm{Nil\arabic*})$}]
\item\label{cond:nil1} $\lambda=0$, $D=0$. This is the Ricci-flat case, examples of which exist in abundance (see e.g.~\cite{ContidelBarcoRossi:SigmaType}).
\item\label{cond:nil2} $\lambda=0$, $D\neq0$. In this case, notice that $D$ is not semisimple, and therefore not a multiple of a Nikolayevsky derivation. Indeed, Theorem~\ref{thm:nilsolitonsandnik} forces $D$ to be nilpotent.
\item\label{cond:nil3} $\lambda\neq0$, $D=0$. This is the Einstein case, studied e.g. in \cite{ContiRossi:EinsteinNilpotent}.
\item\label{cond:nil4} $\lambda\neq0$, $D\neq0$. This is the case that resembles most the Riemannian situation. In this case, if $D$ is semisimple, it is a multiple of a Nikolayevsky derivation. However, since we do not assume the metric to be positive definite, there is no general reason to assume that $D$ is semisimple.
\end{enumerate}

\begin{remark}
An indefinite nilsoliton such that $\Tr D\neq 0$ must belong to~\ref{cond:nil4}. In fact, if $\Tr D\neq 0$ then $D$ is not a nilpotent derivation, hence the nilsoliton must correspond to the second case in  Theorem~\ref{thm:nilsolitonsandnik}, i.e. with $\lambda\neq0$. In conclusion, $D\neq0$ and $\lambda\neq0$ correspond to the case~\ref{cond:nil4}.
\end{remark}

As an example of the condition~\ref{cond:nil2}, we have the following:
\begin{example}\label{ex:Nil2dim4}
Take $\g=(0,0,e^{12},e^{13})$, with $g=g_1e^1\otimes e^1+g_2e^2\odot e^4+g_3e^3\otimes e^3$.

Then
\[\ric = \frac1{2g_1} \left(\frac{g_2^2}{g_3}-g_3\right)e^2\otimes e^2\qquad \text{and}\qquad
\Ric = \frac1{2g_1} \left(\frac{g_2}{g_3}-\frac{g_3}{g_2}\right)e^2\otimes e_4,\]
which is a derivation.

This is a solution with $\lambda=0$.
\end{example}

An example of the~\ref{cond:nil4} case is the Heisenberg Lie algebra  \cite{Onda:ExampleAgebraicRicciSolitons} (see also
Example~\ref{ex:heisenberg}); in that case, the derivation $D$ is semisimple.

It is natural to ask whether in case~\ref{cond:nil4} the derivation $D$ is necessarily semisimple. This turns out to be false, as shown in the following example.
\begin{example}
\label{ex:nil4notss}
Consider the $7$-dimensional Lie algebra
\[\textnormal{257H}:\quad(0,0,0,0,e^{12},e^{34},e^{13}+e^{25}),\]
where the label \textnormal{257H} refers to the classification of~\cite{Gong}. If we take the metric
\begin{gather*}
\frac{g_5^2}{g_7}e^1\otimes e^1 +  \frac{2 g_7}{3 g_5}e^2\otimes e^2 +  e^3\odot e^4 + g_5 e^5\otimes e^5 -\frac{3}{2}e^6\otimes e^6 + g_7e^7\otimes e^7,
\end{gather*}
we obtain
\[\Ric=
\begin{pmatrix}
-\frac{1}{3}&0&0&0&0&0&0\\
0&-\frac{2}{3}&0&0&0&0&0\\
0&0&-\frac{1}{3}&0&0&0&0\\
0&0&-\frac{1}{2} \frac{g_5^{2}}{g_7^{2}}&-\frac{1}{3}&0&0&0\\
0&0&0&0&0&0&0\\0&0&0&0&0&\frac{1}{3}&0\\
0&0&0&0&0&0&\frac{1}{3}
\end{pmatrix}
=-\id+D,\]
where $D$ is a non-semisimple derivation.
\end{example}

\begin{remark}
The Lie algebra of Example~\ref{ex:Nil2dim4} admits nilsoliton metrics of types~\ref{cond:nil1} (see~\cite{ContidelBarcoRossi:SigmaType}),~\ref{cond:nil2} (the one in the example) and~\ref{cond:nil4} (see~\cite{Lauret:Finding}).

The Lie algebra of Example~\ref{ex:nil4notss} admits a
nilsoliton metric of type~\ref{cond:nil1}
(see~\cite{ContidelBarcoRossi:SigmaType}) and a nilsoliton metric of type~\ref{cond:nil4}, namely the indefinite one appearing in the examples. Note that it cannot be a Riemannian nilsoliton (see~\cite{FernandezCulma}).

Kondo and Tamaru~\cite{KondoTamaru:LorentzianNilpotent} constructed $6$ different nilsoliton Lorentzian metrics on the same Lie algebra (see also Example~\ref{ex:KondoTamaruExtensions}).

This shows that the uniqueness of nilsoliton metrics up to scaling and automorphisms as proved in~\cite{Lauret:RicciSoliton} does not extend beyond the Riemannian setting.
\end{remark}

The situation of Example~\ref{ex:nil4notss} turns out to be quite rare. Indeed, few Lie algebras admit a nonsemisimple derivation $D$ satisfying~\eqref{eqn:nik}, i.e. $\Tr D\circ X=\Tr X$ for all derivations $X$. This includes all Lie algebras on which all derivations have zero trace, where every inner derivation $D=\ad v$ satisfies~\eqref{eqn:nik}; in dimension $\leq 7$, this amounts to nine Lie algebras and two one-parameter families (see \cite[Table~1]{ContiRossi:EinsteinNilpotent}). Beside these, there are exactly eight nilpotent Lie algebras of dimension $\leq 7$ with a nonsemisimple derivation satisfying~\eqref{eqn:nik}:
\begin{proposition}
\label{prop:nonil4}
Let $\g$ be a nilpotent Lie algebra of dimension $\leq 7$ such that not all derivations are trace-free. Then $\g$ admits a non-semisimple derivation $D$ such that
$\Tr D\circ X=\Tr X$ for all derivations $X$ if and only if $\g$ is isomorphic to one of the following:
\begin{align*}
\textnormal{12357B}&:\quad(0,0,0,e^{12},e^{14}+e^{23},e^{15}-e^{34},e^{16}+e^{23}-e^{35})\\
\textnormal{12357B1}&:\quad(0,0,0,e^{12},e^{14}+e^{23},e^{15}-e^{34},e^{16}-e^{23}-e^{35})\\
\textnormal{12457B}&:\quad(0,0,e^{12},e^{13},0,e^{14}+e^{25},e^{16}+e^{35}+e^{25})\\
\textnormal{12457K}&:\quad(0,0,e^{12},e^{13},e^{23},e^{24}+e^{15},e^{14}+e^{16}+e^{34})\\
\textnormal{13457G}&:\quad(0,0,e^{12},e^{13},e^{14},e^{23},e^{16}+e^{25}-e^{34}+e^{24})\\
\textnormal{1357L}&:\quad(0,0,e^{12},0,e^{13}+e^{24},e^{14},\frac{1}{2} e^{34}+\frac{1}{2} e^{26}+e^{15}+e^{23})\\
\textnormal{147D}&:\quad(0,0,0,e^{12},e^{23},- e^{13},e^{26}+e^{16}+e^{15}-2 e^{34})\\
\textnormal{257H}&:\quad(0,0,0,0,e^{12},e^{34},e^{13}+e^{25}),
\end{align*}
where labels refer to the classification of~\cite{Gong}.
\end{proposition}
\begin{proof}
Arguing as in Theorem~\ref{thm:nilsolitonsandnik}, we see that the Jordan decomposition of a derivation $D$ satisfying~\eqref{eqn:nik} is $D=D_s+D_n$, where $D_s$ is a Nikolayevsky derivation and $D_n$ is in the null space $\lie n$ of $\langle,\rangle_{\Tr}$. In other words, $D_n$ is in $Z(D_s)\cap\lie n$, where $Z(D_s)$ denotes the centralizer of $D_s$.

Nilpotent Lie algebras of dimension $\leq 7$ are classified in~\cite{Magnin,Gong}. In order to compute the Nikolayevsky derivation $N$ and the space $\lie n$, we used the ad-hoc computer program \cite{gleipnir}. The computation is straightforward for the Lie algebras that do not depend on a parameter; some extra work is required to handle the one-parameter families, since the space $\Der\g$ may depend on the parameter.

We find that for each Lie algebra listed in the statement the centralizer of $N$ intersects $\lie n$ nontrivially, giving rise to nonsemisimple derivations $D$ satisfying~\eqref{eqn:nik}. Notice that $N$ is diagonal relative to the basis in which the Lie algebra is given. We obtain Table~\ref{table:Nintersectn}.

For all other cases the centralizer of $N$ intersects $\lie n$ trivially. Since the Nikolayevsky derivation is unique up to automorphism, this shows that every $D$ satisfying~\eqref{eqn:nik} is semisimple.
\end{proof}

\begin{table}[thp]
\centering
{\caption{\label{table:Nintersectn} Lie algebras with  $Z(N)\cap\lie{n}$ nontrivial}
\begin{scriptsize}
\begin{tabular}{c C C}
\toprule
$\g$ & N & Z(N)\cap\lie{n}\\
\midrule
  12357B & \multirow{2}{*}{$(0, 1, 0, 1, 1, 1, 1)$} &e^2\otimes e_7, e^2\otimes e_4+e^4\otimes e_5+e^5\otimes e_6+e^6\otimes e_7,\\
  12357B1 & & e^2\otimes e_6+e^4\otimes e_7, e^2\otimes e_5+e^4\otimes e_6+e^5\otimes e_7\\[3pt]
  \multirow{3}{*}{12457B} & \multirow{3}{*}{$(0,1,1,1,0,1,1)$} &  e^1\otimes e_5+e^2\otimes (2e_4+e_6)+e^3\otimes e_6, e^2\otimes e_6+e^3\otimes e_7,\\
  &&-e^1\otimes e_5-e^2\otimes (e_4+e_6)+e^4\otimes e_7, e^2\otimes e_7,\\
  &&e^2\otimes e_3 + e^3\otimes e_4+e^4\otimes e_6+e^6\otimes e_7  \\[3pt]
  12457K & (1/2, 0, 1/2, 1, 1/2, 1, 3/2) &  e^1\otimes e_5, -e^1\otimes e_3+e^3\otimes e_5+e^4\otimes e_6\\[3pt]
  \multirow{2}{*}{13457G} & \multirow{2}{*}{$(0, 2/3, 2/3, 2/3, 2/3, 4/3, 4/3)$} &  e^2\otimes e_5, e^2\otimes e_4+e^3\otimes e_5+2e^6\otimes e_7,\\
  & & e^2\otimes (2e_3-e_4)+e^3\otimes (2e_4-e_5)+2e^4\otimes e_5 \\[3pt]
  1357L & 1/17 (5, 10, 15, 10, 20, 15, 25) & e^2\otimes e_4+e^3\otimes e_6\\[3pt]
  147D & (1/2, 1/2, 1/2, 1, 1, 1, 3/2) & e^1\otimes e_2 - e^6\otimes e_5\\[3pt]
  257H & (2/3,1/3,2/3,2/3,1,4/3,4/3) & e^3\otimes e_4\\
\bottomrule
\end{tabular}
\end{scriptsize}
}
\end{table}

We do not know whether the Lie algebras listed in Proposition~\ref{prop:nonil4} (except \textnormal{257H}) admit a nilsoliton metric of type~\ref{cond:nil4} with $D$ nonsemisimple.
However, note that the Lie algebra \textnormal{257H} is \emph{nice}, i.e. it admits a basis $\{e_i\}$ with dual basis $\{e^i\}$ such that each $[e_i,e_j]$, $e_i\hook de^j$ is a multiple of a basis element. Comparing with the classification of~\cite{ContiRossi:Construction}, we see that \textnormal{257H} is the only nice Lie algebra in the list of Proposition~\ref{prop:nonil4}. We obtain:
\begin{corollary}
The Lie algebra \textnormal{257H} (or~\texttt{731:8} in the notation of~\cite{ContiRossi:Construction}) is the only nice nilpotent Lie algebra of dimension $\leq 7$ admitting a pseudo-Riemannian nilsoliton metric of type~\ref{cond:nil4} with nonsemisimple derivation $D$.
\end{corollary}
We are not aware of any example of a nilsoliton of type~\ref{cond:nil4} where the derivation $D$ is nilpotent.

\section{From Einstein solvmanifolds to nilsolitons\textellipsis}
\label{sec:correspondence}
In this section we give a structure theorem for a class of solvable pseudo-Riemannian Einstein Lie algebras in the spirit of~\cite{Heber:noncompact, Wolter:EinsteinSolvable}. Since pseudo-Rie\-mannian geometry has much more flexibility than Riemannian geometry in this respect (see the counterexamples of Sections~\ref{sec:standard} and~\ref{sec:nilsolitons}), in order to obtain our result we need to restrict the class of metrics  considerably. Indeed, we will consider a pseudo-Riemannian analogue of Iwasawa-type Lie algebras.

Given a solvable Lie algebra $\tilde \g$, we will say that a standard decomposition $\tilde\g=\g\oplus^\perp\lie{a}$ is \emph{pseudo-Iwasawa} if
\begin{equation}
 \label{eqn:pseudo-Iwasawa}
 \ad X= (\ad X)^*, \quad X\in\lie{a}.
\end{equation}
Note that the self-adjoint derivation $\ad X$ need not be semisimple, even if the metric is Einstein; see Remark~\ref{rem:EinsteinSolvmanifoldNonSemisimple}.

\begin{example}
Consider the solvable Lie algebra $\tilde{\g}=(0,e^{12},-e^{13})$ with the metric given by
\[e^1\otimes e^1+g_{23}e^2\odot e^3.\]
This metric is Einstein, more precisely Ricci-flat,
and we have a standard decomposition
\[\tilde \g=\g\oplus^\perp \lie{a}, \quad \g=[\tilde{\g},\tilde{\g}]=\Span{e_2,e_3},\; \lie{a}=\Span{e_1},\]
which is pseudo-Iwasawa since $\ad e_1$ is self-adjoint. This Lie algebra is unimodular, so from equation~\eqref{eqn:definitionH} we have that $H=0$.
\end{example}

The following examples show that not all standard decompositions are pseudo-Iwasawa, even if the Einstein condition is satisfied.

\begin{example}
\label{ex:NonIwasawa}
Consider the solvable Lie algebra $\g=(0,-e^{12}+e^{13},-e^{12},0,0,e^{15})$ with neutral metric
\[e^1\odot e^4-g_2e^{2}\otimes e^{2}+g_2e^{2}\odot e^{3}+g_4 e^{5}\odot e^6.\]
It is easy to check that this is a Ricci-flat standard metric for any $g_2,g_4$, we have
\[
\ad e_1=
\begin{pmatrix}
 0 & 0 & 0 & 0 & 0 & 0 \\
 0 & 1 & 0 & 0 & 0 & 0 \\
 0 & 1 & -1 & 0 & 0 & 0 \\
 0 & 0 & 0 & 0 & 0 & 0 \\
 0 & 0 & 0 & 0 & 0 & 0 \\
 0 & 0 & 0 & 0 & -1 & 0 \\
\end{pmatrix},
\qquad
(\ad e_1)^*=
\begin{pmatrix}
 0 & 0 & 0 & 0 & 0 & 0 \\
 0 & -1 & 0 & 0 & 0 & 0 \\
 0 & -1 & 1 & 0 & 0 & 0 \\
 0 & 0 & 0 & 0 & 0 & 0 \\
 0 & 0 & 0 & 0 & 0 & 0 \\
 0 & 0 & 0 & 0 & -1 & 0 \\
\end{pmatrix},
\]
contradicting~\eqref{eqn:pseudo-Iwasawa}.

Since the assumptions of Proposition~\ref{prop:pseudoAzencottWilson} are satisfied, $\g$ is isometric to a different Lie algebra $\hat\g$ admitting a pseudo-Iwasawa decomposition: explicitly, since $(\ad e_1)^s=\ad e_1+(\ad e_1)^*= -2 e^5\otimes e_6$, $\hat\g$ is given by
\[(0,0, 0,0,0,2e^{16}).\]
Notice however that there exist Einstein solvable Lie algebras with a standard decomposition for which it is not possible to apply Proposition~\ref{prop:pseudoAzencottWilson}; see Examples~\ref{ex:AdjointNotDerivationExample1} or~\ref{ex:AdjointNotDerivation}.
\end{example}

\begin{example}
\label{ex:IwasawaOrNot}
Consider the $9$-dimensional solvable lie algebra $\tilde{\g}$ of Example~\ref{ex:derivedalgebradoesnotwork}, with the diagonal metric $\sum g_i e^i\otimes e^i$. A standard decomposition on $\tilde\g$ is determined by the choice of an abelian subalgebra $\Span{e_9}\subset \lie a \subset \Span{e_1,e_2,e_9}$. It is easy to see that $\ad X$ is only symmetric for $X$ a multiple of $e_9$; therefore, there is only one pseudo-Iwasawa decomposition, corresponding to $\lie a=\Span{e_9}$.
\end{example}
\begin{remark}
Example~\ref{ex:IwasawaOrNot} shows that the  pseudo-Iwasawa condition depends both on the metric and the standard decomposition.
\end{remark}

\begin{example}\label{ex:Hnonzeroisotropic}
Consider the solvable Lie algebra $\tilde{\g}=(0,0,e^{13},e^{24})$ with the metric given by
\[e^1\otimes e^1-e^2\otimes e^2+e^3\otimes e^3+e^4\otimes e^4.\]
We have a pseudo-Iwasawa decomposition
\[\tilde \g= [\tilde{\g},\tilde{\g}]\oplus^\perp\lie{a}, \quad \lie{a}=\Span{e_1,e_2},\]
since $\ad e_1=(\ad e_1)^*$ and $\ad e_2=(\ad e_2)^*$. Note that $\Tr\ad e_1=-1=\Tr\ad e_2$. We have $H=-e_1+e_2$ and in this case $\langle H, H\rangle=0$. Note, however, that this metric is not Einstein.
\end{example}

\begin{remark}
Whilst the above Example~\ref{ex:Hnonzeroisotropic} show that $\langle H,H\rangle$ can be zero in a pseudo-Iwasawa Lie algebra, we will be mainly interested in the case where $\langle H,H\rangle\neq0$. This condition is needed in order to apply Proposition~\ref{prop:subalgebraEinstein} and obtain a correspondence between a category of nilsolitons and a category of solvable Einstein Lie algebras. It asserts that $\Tr \ad H$ is nonzero; in particular, this implies that $H$ is not in the center and $\tilde\g$ is not unimodular.

In fact, we will see in Corollary~\ref{cor:Hzero} that Einstein pseudo-Iwasawa Lie algebras with nonzero scalar curvature have either $H=0$ or $\langle H,H\rangle\neq0$.
\end{remark}

In the sequel, we will use the following:
\begin{lemma}\label{lemma:ScalarOfSelfAdjoint}
Let $V$ be a vector space with a (possibly indefinite) scalar product $\langle,\rangle$; let $f,g\colon V\to V$ be linear maps. Then
\[\langle f,g\rangle = \langle f, g^s-g^a\rangle_{\Tr},\]
with $g^s$ and $g^a$ denoting the self-adjoint and anti-self-adjoint part of $g$.
\end{lemma}
\begin{proof}
Let $e_1,\dotsc, e_n$ be an orthonormal basis, i.e. $\langle e_i,e_j\rangle= \epsilon_i \delta_{ij}$ with $\epsilon_i=\pm1$. Then
\begin{align*}
\langle f,g\rangle = \sum_i \epsilon_i \langle f(e_i),g(e_i)\rangle &= \sum_i \epsilon_i\langle g^*(f(e_i)), e_i\rangle \\
&= \sum_i e^i(g^*(f(e_i)))=\Tr (g^s-g^a)\circ f.\qedhere
\end{align*}
\end{proof}

We use Lemma~\ref{lemma:ScalarOfSelfAdjoint} to recover the expression of the Ricci curvature on a solvable Lie algebra with a pseudo-Iwasawa decomposition. Fix an orthogonal basis $e_1,\dotsc, e_k$ of $\lie{a}$ and define the derivations $\phi_s=-\ad e_s$ as in Section~\ref{sec:standard}. We have:
\begin{lemma}
\label{lem:RicciPseudo-Iwasawa} On a metric Lie algebra endowed with a pseudo-Iwasawa decomposition $\tilde{\g}=\g\oplus^{\perp}\lie{a}$, the Ricci tensor of the metric $\langle,\rangle$ on $\tilde \g$ and its restriction to $\g$ are related by:
\begin{align*}
\widetilde\ric(v,w)&=\ric(v,w)
-\langle [H,v],w\rangle,&&v,w\in\g\\
\widetilde \ric(v,X)&=0,&&v\in\g,\ X\in\lie{a}\\
\widetilde \ric(X,Y)&=- \langle \ad X,\ad Y\rangle_{\Tr} && X,Y\in\lie{a}.
\end{align*}
\end{lemma}
\begin{proof}
We use Proposition~\ref{prop:ricciofstandard} to obtain all the above formulas.

For $v,w\in\g$ and using the pseudo-Iwasawa condition  $\phi_s^*=\phi_s$, we get:
\begin{align*}
\widetilde{\ric}(v,w)&=\ric(v,w)+\sum_s\frac12\epsilon_s\langle[\phi_s,\phi_s^*](v), w\rangle  -\frac12\epsilon_s\langle (\phi_s+\phi_s^*)(v),w\rangle \Tr\phi_s\\
&=\ric(v,w)-\sum_s\epsilon_s\langle \phi_s(v),w\rangle\Tr\phi_s,
\end{align*}
where $\sum_s\epsilon_s (\Tr\phi_s)\phi_s=\ad H$ by~\eqref{eqn:derivationphi} of Lemma~\ref{lemma:derivationphi}.

The pseudo-Iwasawa condition allows us to use Lemma~\ref{lemma:ScalarOfSelfAdjoint}; we have that for any $v\in\g$, $e_s\in\lie{a}$:
\[
\widetilde{\ric}(v,e_s)=\frac12\langle \ad v, \phi_s\rangle=\frac12\Tr\ad v \circ \phi_s;
\]
this is zero because $\phi_s$ is a derivation, hence it preserves each term of the lower central series $\g^0=\g$, $\g^{i+1}=[\g,\g^i]$, whilst $\ad v$ maps $\g^k$ to $\g^{k+1}$. Thus, relative to a basis adapted to the lower central series,  $\ad v\circ\phi_s$ is strictly upper triangular, and so it has zero trace.

Finally, using again the pseudo-Iwasawa condition and Lemma~\ref{lemma:ScalarOfSelfAdjoint}, we compute for any $e_s,e_r$ in $\lie{a}$:
\[
\widetilde{\ric}(e_s,e_r)=  - \frac12 \langle \phi_s,\phi_r\rangle -\frac12 \Tr(\phi_s\circ\phi_r)
=- \Tr(\phi_s\circ\phi_r).\qedhere
\]
\end{proof}

It turns out that in the pseudo-Iwasawa context the Einstein condition is strictly related to the existence of a nilsoliton. In particular we have:

\begin{theorem}
\label{thm:solvablenilsolitoncorrespondence}
Let $\tilde\g=\g\oplus^{\perp}\lie{a}$ be a pseudo-Iwasawa decomposition. Then the Einstein equation $\widetilde\Ric = \lambda \id$ on $\tilde\g$ is satisfied if and only if
\begin{enumerate}
\item $\g$ with the induced metric satisfies the nilsoliton equation \[\Ric = \lambda \id+D, \qquad D=\ad H.\]
\item $\langle \ad X, \ad Y\rangle_{\Tr}=-\lambda\langle X,Y\rangle$ for all $X,Y\in\lie a$.
\end{enumerate}
In this case, then
\[\Tr D^2 = -\lambda\Tr D.\]
\end{theorem}
\begin{proof}
By Lemma~\ref{lem:RicciPseudo-Iwasawa},
the Einstein condition is equivalent to
\begin{align*}
\lambda \langle v,w\rangle&=\ric(v,w)-\langle [H,v],w\rangle =\langle \Ric(v)-\ad(H)v,w\rangle, && v,w\in\g,\\
\lambda \langle X,Y\rangle &= -\langle \ad X,\ad Y\rangle_{\Tr}, && X,Y\in\lie a.
\end{align*}
Substituting $D=\ad H$ in the first equation gives $\lambda v=\Ric v - Dv$, i.e. $\Ric=\lambda\id+D$.

The last claim follows from Corollary~\ref{cor:TracciaDquadro}.
\end{proof}

\begin{remark}\label{rm:TrDerivazioneQuadratoEsplicita}
In the pseudo-Iwasawa Einstein setting we can prove explicitly that $\Tr D^2 = -\lambda\Tr D$.
Using equation~\eqref{eqn:definitionH}, part 1 of Theorem~\ref{thm:solvablenilsolitoncorrespondence} and Lemma~\ref{lemma:derivationphi}, we have
\[\Tr D=\Tr\ad H=\langle H,H\rangle =\sum \epsilon_s(\Tr\phi_s)^2.\]
By the expression of $\ad H$ from Lemma~\ref{lemma:derivationphi} we have:
\[\Tr D^2 =\Tr(\ad H)^2= \sum_{r,s} \epsilon_r\epsilon_s (\Tr \phi_s)(\Tr \phi_r)(\Tr \phi_s\circ\phi_r)\]
and by part 2 of Theorem~\ref{thm:solvablenilsolitoncorrespondence} we get
\[\Tr D^2 =-\sum_{r,s} \epsilon_s(\Tr\phi_s)(\Tr\phi_r)\delta_{r,s}\lambda=-\sum_s \epsilon_s(\Tr\phi_s)^2\lambda=-\lambda\Tr\ad H=-\lambda\Tr D.\]
\end{remark}

\begin{corollary}
\label{cor:nilradical}
Let $\tilde \g$ be an Einstein Lie algebra of nonzero scalar curvature with a pseudo-Iwasawa decomposition $\g\oplus^\perp\lie a$. Then $\g$ is the nilradical of $\tilde\g$.
\end{corollary}
\begin{proof}
By construction $\g$ is contained in the nilradical; we must prove the opposite inclusion.

Given $X\in\lie a$ such that $\ad X$ is nilpotent, we have
\[0=\Tr(\ad X\ad Y)= -\lambda\langle X,Y\rangle, \quad Y\in\lie a.\]
Thus, $X$ is necessarily zero, and the nilradical coincides with $\g$.
\end{proof}

As a consequence of Theorem~\ref{thm:normal}, in the Riemannian case, standard Einstein Lie algebras are isometric to solvable Lie algebras of Iwasawa type; in particular, the derived Lie algebra $\g$ is always a nilsoliton~(\cite{Heber:noncompact}). This does not apply in the pseudo-Riemannian case, as shown by the following example:
\begin{example}\label{ex:AdjointNotDerivation}
Take $\tilde{\g}=(0,2e^{12},e^{13},3e^{14}+e^{23})$, with the Einstein metric of Example~\ref{ex:AdjointNotDerivationExample1}.
Since $\ad e_1$ is not symmetric, this Lie algebra does not admit a pseudo-Iwasawa decomposition, and Theorem~\ref{thm:solvablenilsolitoncorrespondence} does not apply. Indeed, by Proposition~\ref{prop:ricciofstandard} we have
\[\Ric =  -\frac{12}{g_1} \id - \frac1{2g_1}(\ad e_1+ (\ad e_1)^*),\]
so $\g$ is not a nilsoliton; rather, it verifies a more general equation of the form
\begin{equation}
 \label{eqn:chenonstudieremomai}
\Ric=\lambda \id + D + D^*, \quad D\in\Der\g;
\end{equation}
in the Riemannian setting, a metric satisfying~\eqref{eqn:chenonstudieremomai} is known as a \emph{semialgebraic Ricci soliton}~(\cite{Jablonski1}). We expect that a systematic study of Einstein solvmanifolds which are not of pseudo-Iwasawa type would require a study of solutions of~\eqref{eqn:chenonstudieremomai}.
\end{example}

Another consequence of Theorem~\ref{thm:solvablenilsolitoncorrespondence} is the following:
\begin{corollary}
\label{cor:Hzero}
Given a pseudo-Iwasawa solvable Lie algebra $\tilde\g=\g\rtimes\lie a$ satisfying $\widetilde\Ric=\lambda \id$ for some nonzero $\lambda$, then either:
\begin{enumerate}
\item  $\tilde\g$ is unimodular, $H=0$ and $\g$ is a nilsoliton of type~\ref{cond:nil3}, with $\Ric=\lambda \id$; or
\item $\tilde\g$ is not unimodular, $\langle H,H\rangle\neq0$, $\g\oplus\Span{H}$ is also Einstein with a pseudo-Iwasawa decomposition, and $\g$ is a nilsoliton of type~\ref{cond:nil4}, with $\Ric=\lambda\id+D$ and $\Tr D\neq0$.
\end{enumerate}
\end{corollary}
\begin{proof}
We first show that either $H=0$ or $\langle H,H\rangle\neq0$. We have that $\lambda^{-1}D$ satisfies~\eqref{eqn:nik}, and by Theorem~\ref{thm:nilsolitonsandnik} the eigenvalues of $D$ are rational. Corollary~\ref{cor:TracciaDquadro} implies that
\[
\Tr(D^2)=-\lambda\Tr(D);
\]
thus $\Tr D$ vanishes if and only if all its eigenvalues are zero if and only if $N=0$. In other words, $\langle H,H\rangle=0$ is equivalent to the vanishing of the Nikolayevsky derivation of $\g$. This implies that all derivations are traceless; in particular, $\g$ is unimodular, so $H=0$. By Theorem~\ref{thm:solvablenilsolitoncorrespondence}, $\g$ is Einstein with $\Ric=\lambda \id$.

If $\langle H,H\rangle\neq0$, Proposition~\ref{prop:subalgebraEinstein} implies that $\g\rtimes\Span{H}$ is Einstein; since $\ad H$ is self-adjoint, the decomposition is pseudo-Iwasawa. In addition, $\Ric=\lambda \id +D$ on $\g$, with $D=\ad H$, and by construction $\Tr D=\langle H,H\rangle$ is nonzero.
\end{proof}

\begin{corollary}
\label{cor:restrictricciflat}
If $\tilde\g$ is Ricci-flat and it has a pseudo-Iwasawa decomposition $\tilde\g=\g\oplus^\perp\lie{a}$, then $\langle,\rangle_{\Tr}$ is zero on $\lie a$ and either:
\begin{enumerate}
\item $\tilde \g$ is unimodular and $\g$ is a nilsoliton of type~\ref{cond:nil1};
\item $\tilde \g$ is nonunimodular, $\ad H=0$ and $\g$ is a nilsoliton of type~\ref{cond:nil1};
\item $\tilde \g$ is nonunimodular, $\ad H\neq 0$ and $\g$ is a nilsoliton of type~\ref{cond:nil2}.
\end{enumerate}
Moreover, if $\tilde \g$ is nonunimodular then $\lie a$ contains a two-dimensional nondegenerate subspace $\Span{H,X}$, with $\Tr\ad X\neq0$ and $\Tr (\ad X)^2=0$; the restriction to $\g\rtimes\Span{H,X}$ is also Ricci-flat.
\end{corollary}
\begin{proof}
By Theorem~\ref{thm:solvablenilsolitoncorrespondence} $\langle,\rangle_{\Tr}$ is zero on $\lie a$ and  $\Ric=D=\ad H$, so $\g$ is a nilsoliton of type~\ref{cond:nil1} or~\ref{cond:nil2} according to whether $\ad H$ is zero or not.

If $H$ is nonzero, $\tilde \g$ is nonunimodular. The derivation $D$ is traceless by Theorem~\ref{thm:nilsolitonsandnik}, so $\lie a$ must contain some other $X$ independent of $H$ with $\Tr \ad X\neq0$. By definition of $H$ we have
\[\langle H,H\rangle=0, \quad \Tr\ad X=\langle H,X\rangle\neq0,\]
so $\Span{H,X}$ is nondegenerate. The subalgebra $\g\rtimes\Span{H,X}$ is again Ricci-flat by Theorem~\ref{thm:solvablenilsolitoncorrespondence}.
\end{proof}

In Corollary~\ref{cor:restrictricciflat}, if $\lie a$ has dimension one, $\tilde \g$ is unimodular and only the geometry~\ref{cond:nil1} can occur. When $\tilde \g$ is nonunimodular, both \ref{cond:nil1} and \ref{cond:nil2} can occur, as evident from the Examples~\ref{ex:estendonil1} and~\ref{ex:estendonil2}.

\begin{example}
\label{example:nil2nonsiestende}
The nilsoliton of Example~\ref{ex:Nil2dim4} admits only one self-adjoint derivation independent from $D=e^2\otimes e_4$, namely $X=e^2\otimes e_2+e^3\otimes e_3+e^4\otimes e_4$. Since $\langle X,X\rangle_{\Tr{}}\neq0$, this nilsoliton cannot be obtained as a restriction of a Ricci-flat pseudo-Iwasawa solvable Lie algebra.
\end{example}

\section{\textellipsis and back again}
In this section we focus on nilsolitons and describe how we can extend them to Einstein solvable Lie algebras of the form $\g\oplus^\perp\lie a$. The construction is already known for the case with $\lie a$ of dimension one (see \cite{Yan:Pseudo-RiemannianEinsteinhomogeneous,YanDeng:DoubleExtensionRiemNilsolitons}), but we will work more generally, and recover the one-dimensional results as a special case.

As in the Riemannian case (see \cite[p.~313]{Heber:noncompact}), nilsolitons of type~\ref{cond:nil3} and~\ref{cond:nil4} can be used to obtain Einstein metrics in higher dimensions:
\begin{theorem}
\label{thm:constructionstandardextension}
Let $\g$ be a nilsoliton, $\Ric=\lambda \id+D$, $\lambda\neq0$. Let $\lie a\subset\Der\g$ be a  subalgebra of self-adjoint derivations containing $D$, and assume that $\langle,\rangle_{\Tr}$ is nondegenerate on  ${\lie a}$. Then
the metric $\langle,\rangle_{\g}-\frac1{\lambda}\langle,\rangle_{\Tr}$ on $\tilde \g=\g\rtimes\lie a$ is Einstein with $\widetilde\Ric=\lambda \id$ and the decomposition $\tilde\g=\g\oplus^\perp\lie{a}$ is pseudo-Iwasawa.
\end{theorem}
\begin{proof}
The Lie algebra $\lie{a}$ is necessarily abelian, as the commutator of two self-adjoint maps is skew-self adjoint; therefore, $\tilde\g=\g\oplus^\perp\lie{a}$ is a standard decomposition.

By Lemma~\ref{lemma:ScalarOfSelfAdjoint}, the second condition of Theorem~\ref{thm:solvablenilsolitoncorrespondence} reads $\langle \phi_s,\phi_r\rangle_{\Tr}=-\lambda \langle \phi_s,\phi_r\rangle$; the statement follows.
\end{proof}

We will call the solvable Lie algebra constructed in Theorem~\ref{thm:constructionstandardextension} a \emph{pseudo-Iwasawa extension} of the nilsoliton $\g$.

\begin{remark}\label{rem:NoNilpotentDerivation}
Theorem~\ref{thm:constructionstandardextension} requires nondegeneracy of $\langle,\rangle_{\Tr}$ restricted to $\lie{a}$. If $\ad\lie{a}$  consists of nilpotent elements, then $\langle,\rangle_{\Tr}$ is zero on $\lie a$ by Engel's theorem.

In particular, the Lie algebra $\tilde\g$ is only nilpotent in the trivial case where $\g$ is Einstein and $\lie{a}=0$.
\end{remark}

\begin{example}
\label{ex:ExtensionNilsoliton}
Consider the Lie algebra
$\g=(0,0,0,e^{12},e^{13})$ that admits the nilsoliton metric
\[\langle,\rangle=e^1\otimes e^1+e^2\odot e^3-\frac{1}{2}e^4\odot e^5\]
satisfying $\Ric=\id-D$ where $D=\diag\left(\frac{1}{2},\frac{3}{4},\frac{3}{4},\frac{5}{4},\frac{5}{4}\right)$ is the Nikolayevsky derivation of $\g$. The generic symmetric derivation is
\[\left(\begin{array}{ccccc}\lambda_4-\lambda_1&0&0&0&0\\0&\lambda_1&\lambda_2&0&0\\0&\lambda_3&\lambda_1&0&0\\0&0&0&\lambda_4&\lambda_2\\0&0&0&\lambda_3&\lambda_4\end{array}\right).\]
In order to find a nondegenerate Lie algebra of symmetric derivations, we can for instance set $\lambda_2=\lambda_3$, and write
\begin{gather*}
\lie{a}=\Span{e_6,e_7,e_8}, \quad
\ad e_6=-e^1\otimes e_1+e^2\otimes e_2+e^3\otimes e_3,\\
\ad e_7=e^1\otimes e_1+e^4\otimes e_4+e^5\otimes e_5,\quad
\ad e_8=e^2\otimes e_3+e^3\otimes e_2+e^4\otimes e_5+e^5\otimes e_4.
\end{gather*}
Then $\tilde\g=\g\rtimes\lie{a}$ takes the form
\[(-e^{16}+e^{17},e^{26}+e^{38},e^{36}+e^{28},e^{12}+e^{47}+e^{58},e^{13}+e^{57}+e^{48},0,0,0),\]
with the Einstein metric
\[e^1\otimes e^1+e^2\odot e^3-\frac{1}{2}e^4\odot e^5-3e^6\otimes e^6+e^6\odot e^7-3e^7\otimes e^7 -4e^8\otimes e^8.\]
Inside $\g\rtimes\lie{a}$, one finds the Einstein Lie algebra with a pseudo-Iwasawa decomposition
\[\Span{e_1,e_2,e_3,e_4,e_5,\frac34e_6+\frac54e_7}\cong \g\rtimes\Span{D},\]
or $\g\oplus\Span{H}$ in the notation of Proposition~\ref{prop:subalgebraEinstein}.
Notice that any nondegenerate $\hat\g$ with
\[\g\rtimes\Span{D}\subset \hat\g \subset \g\rtimes\lie a\]
is again an Einstein Lie algebra with a pseudo-Iwasawa decomposition.
\end{example}
\begin{example}
In the construction of Theorem~\ref{thm:constructionstandardextension},
the choice of $\lie a$ is not unique. For instance in Example~\ref{ex:ExtensionNilsoliton} we could for instance have chosen
\[\ad e_8=e^2\otimes e_3-e^3\otimes e_2+e^4\otimes e_5-e^5\otimes e_4.\]
The resulting Lie algebras $\tilde\g$ are not isomorphic, as only one is completely solvable.

However, $\lie a$ is not allowed to contain the self-adjoint nilpotent derivation $\hat{D}=e^{2}\otimes e_3+e^4\otimes e_5$, because $\langle \hat D,\phi\rangle_{\Tr}=0$ for any self-adjoint derivation $\phi$ commuting with $\hat D$. Indeed, in this case we can choose for $\lie a$ any space of self-adjoint derivations which contains $D$ but not $\hat D$.

Thus, among the possible choices for $\lie a$ there is a minimal subalgebra, namely $\Span D$, but not a maximal one.
\end{example}

\begin{example}\label{ex:extensionNONpseudo-Iwasawa}
Consider the $8$-dimensional nice Lie algebra:
\[(0, 0, 0, e^{12}, e^{13},e^{24}, e^{15} + e^{23},e^{26} + e^{14})\]
which admits the nilsoliton metric
\begin{multline*}
\langle,\rangle=\frac{1}{10}\sqrt[3]{\frac{3}{2}}e^1\otimes e^1+\frac{1}{10} \sqrt[3]{\left(\frac{3}{2}\right)^{2}}e^2\otimes e^2+g_3e^3\otimes e^3+\frac{9}{100}e^4\otimes e^4\\
-g_3\sqrt[3]{\frac{3}{2}}e^5\otimes e^5
-\frac{9}{125} \sqrt[3]{\left(\frac{3}{2}\right)^{2}}e^6\otimes e^6+g_3\sqrt[3]{\left(\frac{3}{2}\right)^{2}}e^7\otimes e^7+\frac{27}{250}\sqrt[3]{\frac{3}{2}}e^8\otimes e^8,
\end{multline*}
satisfying
\[\Ric=\id-D\qquad D=e^3\otimes e_3 + e^5\otimes e_5 + e^7\otimes e_7.\]

Since $\Tr D^2=\Tr D=3\neq0$,  we can apply Theorem~\ref{thm:constructionstandardextension} to obtain an Einstein metric on the solvable Lie algebra $\tilde\g=\g\rtimes_D \Span{e_9}$. This Lie algebra was also considered in Examples~\ref{ex:derivedalgebradoesnotwork} and~\ref{ex:IwasawaOrNot}.

Notice that in this case all symmetric derivations are multiples of $D$.
\end{example}

\begin{corollary}
\label{cor:uniquenessstandardextension}
Let $\tilde \g=\g\oplus^\perp\lie a$ be an Einstein solvable Lie algebra of pseudo-Iwasawa type of nonzero scalar curvature. Up to isometric isomorphisms, $\tilde\g$ is a pseudo-Iwasawa extension of $\g$.
\end{corollary}
\begin{proof}
The map $\ad\colon\lie a\to\Der\g$ is injective by  Theorem~\ref{thm:solvablenilsolitoncorrespondence}, as we are assuming $\lambda\neq0$. This effectively identifies $\lie a$ with an algebra of symmetric derivations. The metric is determined by Theorem~\ref{thm:solvablenilsolitoncorrespondence}.
\end{proof}

Recall from Corollary~\ref{cor:Hzero} that a nilsoliton of type~\ref{cond:nil4} with $\Tr D=0$ cannot be extended to an Einstein pseudo-Iwasawa Lie algebra. If we  impose $\Tr D\neq0$ and specialize to the case with $\lie a$ of dimension one, we recover the result of~\cite[Theorem~4.7]{Yan:Pseudo-RiemannianEinsteinhomogeneous}:
\begin{corollary}
\label{cor:extendNilsolitonNil4}
Let $\g$ be a nilpotent Lie algebra with a metric $\langle,\rangle$ satisfying $\Ric=\lambda \id+D$, with $D$ a derivation with nonzero trace. Setting
\[\tilde\g=\g\rtimes_D\Span{e_0}, \quad \widetilde{\langle,\rangle}=\langle,\rangle+ (\Tr D) e^0\otimes e^0\]
defines an Einstein pseudo-Iwasawa extension of $\g$ with $\widetilde\Ric=\lambda \id$.
\end{corollary}
\begin{proof}
By Theorem~\ref{thm:nilsolitonsandnik} $\lambda$ is nonzero, so $\g$ is indeed a nilsoliton of type~\ref{cond:nil4}. By Corollary~\ref{cor:TracciaDquadro} we have $\langle D,D\rangle_{\Tr}=-\lambda \Tr D\neq0$, and $D=\Ric-\lambda \id$ is self-adjoint. Applying Theorem~\ref{thm:constructionstandardextension} with $\ad e_0=D$ we obtain the Einstein metric
\[\langle,\rangle -\frac1\lambda \langle D,D\rangle_{\Tr} e^0\otimes e^0=\langle,\rangle+(\Tr D)e^0\otimes e^0.\qedhere\]
\end{proof}

\begin{remark}\label{rem:SignatureExtension}
In Corollary~\ref{cor:extendNilsolitonNil4}, $\Tr D$ and $\lambda$ have opposite signs; indeed $\Tr D$ is nonzero, so the Nikolayevsky derivation $N$ is nonzero; therefore, by Theorem~\ref{thm:nilsolitonsandnik}, $\Tr D^2=\Tr N^2>0$ and Corollary~\ref{cor:TracciaDquadro} gives $\Tr D=-\Tr D^2/\lambda$. Therefore, if the nilsoliton metric has signature $(p,q)$, the Einstein metric on $\tilde\g$ has signature $(p+1,q)$ when the nilsoliton is \emph{shrinking} ($\lambda<0$) and $(p,q+1)$ when it is \emph{expanding} ($\lambda>0$). We will see in Example~\ref{ex:KondoTamaruExtensions} that for a fixed Lie algebra and  signature we can have both expanding and shrinking nilsolitons.
\end{remark}

\begin{remark}\label{rem:EinsteinSolvmanifoldNonSemisimple}
Unlike in the Riemannian case, on a pseudo-Riemannian Einstein solvmanifold with a pseudo-Iwasawa decomposition, the derivation $\ad H$ is not necessarily semisimple. An example can be obtained by applying Corollary~\ref{cor:extendNilsolitonNil4} to Example~\ref{ex:nil4notss}.
\end{remark}

\begin{remark}
\label{rem:JoiningMoreNilsolitons}
Passing from a nilsoliton to its Einstein extension does not generally preserve reducibility: given irreducible nilsolitons $\g_1,\dots, \g_k$ of type~\ref{cond:nil4}, the pseudo-Iwasawa extension $\tilde{\g}=(\g_1\oplus\dots \oplus \g_k)\rtimes\R$ is generally irreducible. See next example.
\end{remark}

\begin{example}
Consider the Heisenberg Lie algebra $\g_1=(0,0,e^{12})$ with the nilsoliton metric $g_{1}e^1\odot e^2 + \frac{2}{3}g_{1}^2e^3\otimes e^3$, for which
\[\Ric=\id+D_1, \qquad D_1= -\frac{2}{3}(e^1\otimes e_1+e^2\otimes e_2+2 e^3\otimes e_3),\]
and the Lie algebra $\g_2=(0,0,\hat{e}^{12},\hat{e}^{13})$ with the  metric $\hat{g}_{1}\hat{e}^1\otimes \hat{e}^1 + \hat{g}_2 \hat{e}^2\otimes \hat{e}^2 - \frac{2}{3}\hat{g}_{1} \hat{g}_{2}\hat{e}^3\otimes \hat{e}^3+\frac{4} {9}\hat{g}_{1}^2\hat{g}_2\hat{e}^4\otimes \hat{e}^4$, which is a nilsoliton with
\[\Ric=\id+D_2, \qquad D_2= -\frac{1}{3}(\hat{e}^1\otimes \hat{e}_1+2\hat{e}^2\otimes \hat{e}_2+3 \hat{e}^3\otimes \hat{e}_3+4 \hat{e}^4\otimes \hat{e}_4).\]
Notice that $\lambda$ has been normalized to one for both metrics.

Then, using Remark~\ref{rem:JoiningMoreNilsolitons} we can construct an irreducible Einstein solvable Lie algebra $\tilde\g=(\g_1\oplus\g_2)\rtimes e_0\R$ by setting $[e_0,X_1+X_2]=D_1(X_1)+D_2(X_2)$ for  $X_i\in\g_i$, together with the left-invariant metric given by:
\begin{multline*}
-6e^0\otimes e^0+g_{1}e^1\odot e^2 + \frac{2}{3}g_{1}^2e^3\otimes e^3 \\+\hat{g}_1e^4\otimes e^4 + \hat{g}_2 e^5\otimes e^5 - \frac{2}{3}\hat{g}_{1} \hat{g}_{2}e^6\otimes e^6+\frac{4}{9}\hat{g}_{1}^2 \hat{g}_2 e^7\otimes e^7,
\end{multline*}
where $\hat{e}_i$ is identified with $e_{4+i}$.
\end{example}

An explicit construction similar to Corollary~\ref{cor:extendNilsolitonNil4} applies in the case~\ref{cond:nil3}, i.e. for Einstein nilpotent Lie algebras with non-zero scalar curvature (compare with \cite[Theorem 1.2]{YanDeng:DoubleExtensionRiemNilsolitons}):
\begin{corollary}\label{cor:extendNilsolitonNil3}
Let $\g$ be a nilpotent Lie algebra with an Einstein metric $\langle,\rangle$, with $\Ric=\lambda \id$, $\lambda\neq0$, and let $\psi$ be a self-adjoint derivation of $\g$ with $\Tr\psi^2\neq0$. Setting
\[\tilde\g=\g\rtimes_\psi\Span{e_0}, \qquad \widetilde{\langle,\rangle}=\langle,\rangle-\frac1{\lambda} (\Tr \psi^2) e^0\otimes e^0\]
defines a unimodular pseudo-Iwasawa extension of $\g$ with $\widetilde\Ric=\lambda \id$.
\end{corollary}
\begin{proof}
By \cite[Theorem~4.1]{ContiRossi:EinsteinNilpotent}, all derivations of $\g$ are tracefree, so in particular $\Tr\Psi$ is zero; this implies that $H$ is zero, thanks to~\eqref{eqn:derivationphi}. The extension $\tilde\g$ is Einstein by Theorem~\ref{thm:solvablenilsolitoncorrespondence}, and  unimodular by the definition of $H$.
\end{proof}

An application of this construction can be obtained using diagonal derivations, as we see in the next example.

\begin{example}
\label{ex:einsteindentroeinstein}
Consider the $9$-dimensional nilpotent Lie algebra
\[\g=(0,0,0,0,e^{13}+e^{24},- e^{12},e^{34},e^{15}+e^{23}+e^{46},e^{14}+e^{27}+e^{35}).\]
This admits two diagonal Einstein metrics with $\lambda=1/2$ (see \cite[Example~4.6]{ContiRossi:EinsteinNice}), consider for example the diagonal metric $\langle,\rangle=\sum g_ie^i\otimes e^i$:
\begin{align*}
g_1&=1,\quad g_2=-\frac{3}{16} \left(\sqrt{249}+9 \right), &g_3 &= \frac{731- 47 \sqrt{249}}{2205},\\
g_4&=\frac{131253-8321 \sqrt{249}}{463050 }, &g_5&=\frac{1}{735} \left(47 \sqrt{249}-731\right),\\
g_6&=\frac{9}{16} \left(5 \sqrt{249}+73 \right), &g_7&=\frac{16 \left(333103 \sqrt{249}-5256379\right)}{170170875},\\
g_8&=\frac{2}{105} \left(183-11 \sqrt{249}  \right), &g_9&=\frac{4 \left(131253-8321 \sqrt{249}\right)}{231525}.
\end{align*}
Diagonal derivations are spanned by
$\diag(-1, -2, 1, 2, 0, -3, 3, -1, 1)$.
If we fix
\[\psi=\frac1{\sqrt{60}} \diag({-1, -2, 1, 2, 0, -3, 3, -1, 1}),\qquad \Tr\psi^2 = \frac{1}{2}\]
we can apply Corollary~\ref{cor:extendNilsolitonNil3} and obtain a solvable Lie algebra $\tilde \g=\g\rtimes_\psi\R$ with an Einstein metric. Explicitly, setting $\ad_{e_{10}}=\sqrt{60}\psi$ for simplicity, the structure constants of $\tilde \g$ are
\begin{multline*}
\bigl(- e^{1,10},
-2  e^{2,10},
e^{3,10},
2  e^{4,10},
e^{24}+e^{13},\\
-3 e^{6,10}-e^{12},
3 e^{7,10}+e^{34},
e^{46}+e^{15}+e^{23}-e^{8,10},
e^{27}+e^{35}+e^{14}+e^{9,10},0\bigr).
\end{multline*}
The metric $\widetilde{\langle,\rangle}$ on $\tilde{\g}$ is given by the previous equations (i.e. by $\langle,\rangle$ on $\g$) and $\widetilde{\langle e_{10},e_{10}\rangle}=-\dfrac{\Tr(\ad_{e_{10}}^2)}{\lambda}=-60$. Notice that $\tilde\g$ is unimodular, so in this case $H=0$ and the nilsoliton metric induced on the nilradical according to Corollary~\ref{cor:extendNilsolitonNil3} is actually Einstein.

In particular, this example shows that a nilpotent Einstein manifold can be the nilradical of an Einstein solvmanifold.
\end{example}

For nilsolitons of type~\ref{cond:nil1}, i.e. Ricci-flat metrics, we have the following:
\begin{proposition}
\label{prop:extendnil1}
Let $\g$ be a nilpotent Lie algebra with a Ricci-flat metric $\langle,\rangle$. Let $\lie a\subset\Der\g$ be a subalgebra of self-adjoint derivations and assume that $\langle,\rangle_{\Tr}$ is zero on ${\lie a}$.
\begin{enumerate}
\item If all elements of $\lie a$ are trace-free,
\[\tilde\g=\g\rtimes\lie a,\qquad
 \widetilde{\langle,\rangle}=\langle,\rangle_{\g}+\langle,\rangle_{\lie a}\]
defines a Ricci-flat Lie algebra with a pseudo-Iwasawa decomposition for any metric $\langle,\rangle_{\lie a}$ on $\lie a$.
\item If not all elements of $\lie a$ are trace-free, for any metric $\langle,\rangle_{\lie a}$ on $\lie a$ which restricts to a nondegenerate metric on the subspace $\lie a_0\subset\lie a$ of trace-free elements,
\[\tilde\g=\g\rtimes(\lie a\oplus\Span{H}),\qquad
 \widetilde{\langle,\rangle}=\langle,\rangle_{\g}+\langle,\rangle_{\lie a} + \langle, \rangle_{H}\]
(where $\adtilde H=0$ and $\langle X,H\rangle_H=\Tr X$)
defines a Ricci-flat Lie algebra with a pseudo-Iwasawa decomposition.
\end{enumerate}
\end{proposition}
\begin{proof}
In the first case $\tilde\g$ is unimodular; the statement follows from Theorem~\ref{thm:solvablenilsolitoncorrespondence}.

In the second case, $H$ satisfies~\eqref{eqn:definitionH} and $\adtilde H=0$, so again we conclude using Theorem~\ref{thm:solvablenilsolitoncorrespondence}.
\end{proof}
\begin{remark}
Unlike the situation of Remark~\ref{rem:NoNilpotentDerivation}, the derivation $\psi$ is allowed to be nilpotent in  Proposition~\ref{prop:extendnil1}. Thus, nilpotent Ricci-flat Lie algebras may admit nilpotent Ricci-flat extensions.
\end{remark}

We will explain this procedure in the following example.

\begin{example}
Consider the Heisenberg Lie algebra $\lie{h}=(0,0,e^{12})$ with the Ricci-flat metric $e^1\odot e^3 +e^2\otimes e^2$ and consider the derivations $D_1= e^1\otimes e_2+ e^2\otimes e_3$ and $D_2=e^1\otimes e_3$. These derivations commute, are self-adjoint with respect to $\langle,\rangle$ and satisfy $\Tr (D_i\circ D_j)=0=\Tr D_i$. Thus Proposition~\ref{prop:extendnil1} applies and setting $\ad {e_4}=D_1,\ad {e_5}=D_2$ on the nilpotent Lie algebra
\[\tilde\g=(0,e^{41}, e^{42}+e^{51}+e^{12},0,0)\]
we have a $3$-parameter family of Ricci-flat metrics given by:
\[\widetilde{\langle,\rangle}= e^1\odot e^3 +e^2\otimes e^2+g_4 e^4\otimes e^4+g_{45}e^4\odot e^5+g_5e^5\otimes e^5.\]

Note that the extension of Proposition~\ref{prop:extendnil1} may be solvable nonnilpotent: e.g. on the nilpotent Lie algebra $(0,0,e^{12},e^{12})$ with the Ricci-flat metric $e^1\odot e^3+e^2\odot e^4$ one can consider the self-adjoint derivation \[D'= -e^1\otimes e_1+e^1\otimes e_2-e^2\otimes e_1+e^2\otimes e_2-e^3\otimes e_3-e^3\otimes e_4+e^4\otimes e_3+e^4\otimes e_4,\]
obtaining a unimodular solvable nonnilpotent extension $\g\rtimes_{D'}\R$ with a Ricci-flat metric that extends the Ricci-flat metric on $\g$.
\end{example}

\begin{example}
\label{ex:estendonil1}
Consider $\R^3\rtimes \lie a$, with
\[\ad e_4=e^1\otimes e_2 - e^2\otimes e_1 + \sqrt2 e^3\otimes e_3, \quad \ad e_5=0.\]
Choose
\[\langle,\rangle = e^1\odot e^2 + e^3\otimes e^3+e^4\odot e^5.\]
Then $H=\sqrt2 e_5$. This is a nonunimodular extension of a Ricci-flat metric.
\end{example}

Similarly, for nilsolitons of type~\ref{cond:nil2}, we have the following:
\begin{proposition}
\label{prop:extendnil2}
Let $\g$ be a nilpotent Lie algebra with a nisoliton metric $\langle,\rangle$ of type~\ref{cond:nil2}, $\Ric=D$. Let $\lie a\subset\Der\g$ be a subalgebra of self-adjoint derivations containing $D$ and assume that $\langle,\rangle_{\Tr}$ is zero on ${\lie a}$. Then for any metric $\langle,\rangle_{\lie a}$ on $\lie a$ such that $\langle D,X\rangle_{\lie a}=\Tr X$ for all $X$,
\[\tilde\g=\g\rtimes\lie a,\qquad
 \widetilde{\langle,\rangle}=\langle,\rangle_{\g}+\langle,\rangle_{\lie a}\]
defines a Ricci-flat Lie algebra with a pseudo-Iwasawa decomposition.

Moreover $\dim \lie{a}\geq2$ and $\tilde{\g}$ is nonunimodular.
\end{proposition}
\begin{proof}
The first part is as in Proposition~\ref{prop:extendnil1}; here we only prove the last part.
Since $\langle,\rangle_{\lie{a}}$ is non degenerate and $\langle D,D\rangle_\lie{a}=\Tr D=0$, there must exist at least another derivation $X$ in $\lie{a}$ such that $\Tr X=\langle D,X\rangle_\lie{a}\neq0$, so $\dim \lie{a}\geq2$ and $\tilde{\g}$ is nonunimodular.
\end{proof}

\begin{example}
\label{ex:estendonil2}
Consider the $5$-dimensional nilpotent Lie algebra $\g=\lie{h}\oplus\R^2=(0,0,e^{12},0,0)$ with the metric
\[\langle,\rangle_{\g}=g_1e^1\otimes e^1+g_2e^2\otimes e^2 + 2g_1g_2 e^3\odot e^4+g_5 e^5\otimes e^5,\quad \Ric=D=e^4\otimes e_3\]
which is a nilsoliton of type~\ref{cond:nil2}, for any $g_i$. Consider also the derivation given by
\[X=e^1\otimes e_2-2 e^2\otimes e_1+2e^5\otimes e_5,\]
which commutes with $D$, and satisfies $\Tr{X}=2\neq0$, $\Tr{X^2}=0$. Choosing in the metric $g_2=-2g_1$, the derivation $X$ is also selfadjoint. Hence, we can construct a non unimodular solvable Lie algebra $\tilde\g=\g\rtimes \lie{a}$ where $\lie{a}$ contains the vector $e_6=H$ such that $\ad_H=D$ and a vector $e_7$ such that $\ad_{e_7}=X$. Considering the pseudo-Riemannian metric on $\tilde\g$ given by:
\[\langle,\rangle_{\tilde{\g}}=\langle,\rangle_{\g}+2e^6\odot e^7=g_1e^1\otimes e^1-2g_1e^2\otimes e^2 -4g_1^2 e^3\odot e^4+g_5 e^5\otimes e^5 +2e^6\odot e^7,\]
we have found a Ricci-flat metric on $\tilde{\g}$. Explicitly the Lie algebra $\tilde{\g}$ is given by:
\[(2e^{72},-e^{71},e^{12}-e^{64},0,-2e^{75},0,0).\]
\end{example}

\begin{example}\label{ex:KondoTamaruExtensions}
Kondo and Tamaru proved in~\cite{KondoTamaru:LorentzianNilpotent} that for any $n\geq 4$, there are exactly $6$ left-invariant Lorentzian nilsoliton metrics (up to automorphisms and rescaling) on the product of the Heisenberg Lie algebra $\lie h$ and $\R^{n-3}$, namely on the Lie algebra $\g_n=(0,0,\dots,0,e^{12})$. Those $6$ metrics $g_{\lambda,\xi}$ can be written explicitly with respect to the basis $\{e^i\}$ as
\[
\sum_{i=1}^{n-3}e^i\otimes e^i-\xi e^1\odot e^{n-1}-\lambda e^1\odot e^{n}+(1+\xi^2)e^{n-1}\otimes e^{n-1}+\lambda\xi e^{n-1}\odot e^{n}+(\lambda^2-1)e^{n}\otimes e^{n},
\]
and depend on two parameters $(\lambda,\xi)\in\{(0,0),(1,0),(1,1),(2,0),(2,\sqrt{3}),(2,2)\}$. A direct computation, using the formula for the Ricci in \cite[Section 5]{KondoTamaru:LorentzianNilpotent}, allows us to classify those nilsolitons according to their type. Moreover, we can use the results of this section to determine whether an Einstein pseudo-Iwasawa extension exists.
\begin{enumerate}
\item The metric associated to $(1,0)$ is Ricci-flat (i.e.~\ref{cond:nil1}). We see that, depending on the dimension of $n$, there are many choices of subalgebras $\lie a$ of trace-free self-adjoint derivations and such that $\langle,\rangle_{\Tr}$ is zero. For example, we can choose the derivation $X=e^1\otimes e_{n-1}-e^{n-1}\otimes e_n$ and applying Proposition~\ref{prop:extendnil1} obtain a Ricci-flat metric on the nilpotent Lie algebra $\tilde{\g}_n=\g_n\rtimes_X \R$, with signature $(n,1)$ or $(n-1,2)$ according to the sign chosen for $\langle X,X\rangle$.
\item The metrics associated to $(1,1)$ and $(2,\sqrt{3})$ are nilsolitons of type~\ref{cond:nil2}. However, a straightforward computation shows that in both cases there are no self-adjoint derivations $X$ commuting with $D$ such that $\Tr X^2=0$ and $\Tr X\neq0$, so no extension to a pseudo-Iwasawa Einstein solvmanifold can be found.
\item The metrics associated to $(0,0)$, $(1,\sqrt{3})$ and $(2,2)$ are nilsolitons of type~\ref{cond:nil4}; with reference to~\eqref{eqn:nilsoliton}, the coefficients $\lambda$ are respectively $\frac32$, $\frac{27}{2}$ and $-\frac{9}{2}$. Corollary~\ref{cor:extendNilsolitonNil4} allows the construction of a pseudo-Iwasawa extension $\tilde{\g_n}=\g_n\rtimes_D\R$, with signature $(n-1,2)$ in the first two cases and $(n,1)$ in the last, according to  Remark~\ref{rem:SignatureExtension}.
\end{enumerate}
Since $\g_n$ admits derivations with nonzero trace, there are no metrics of type~\ref{cond:nil3} (see \cite[Theorem~4.1]{ContiRossi:EinsteinNilpotent}).
\end{example}

In light of Proposition~\ref{prop:pseudoAzencottWilson}, it is not surprising that the methods of this section can also be used to obtain Einstein solvable Lie algebras that do not admit a pseudo-Iwasawa decomposition. The idea is that a pseudo-Iwasawa Einstein solvable Lie algebra determines a pseudo-Riemannian manifold which generally admits other simply transitive actions of a solvable Lie group, not necessarily of pseudo-Iwasawa type. A similar technique was used in~\cite{Jablonski1} to obtain solvable Lie algebras with a Ricci soliton metric which are not solvsolitons.

Given a Lie algebra $\g$ with a metric $\langle,\rangle$, choose $\lie a$ such that
\begin{equation}
\label{eqn:AW}
\lie a\subset\Der\g, \quad
[\lie a, \lie a]=0, \quad \lie a^s:=\{f^s\st f\in\lie a\}\subset \Der\g, \quad [\lie a, \lie a^s]=0.
\end{equation}
We can define a scalar product on $\lie a$ by
\[\langle f,g\rangle_{s}=\langle f^s,g^s\rangle_{\Tr{}}.\]

\begin{proposition}
\label{prop:costruiscinonpseudoiwasawa}
Let $\g$ be a nilsoliton, $\Ric=\lambda\id+D$, and let $\lie a$ be a subspace of $\Der\g$ satisfying~\eqref{eqn:AW}.
\begin{enumerate}
\item \label{item:nonpseudoIwasawalambdanonzero}
If $\lambda\neq0$, $D\in\lie a^s$ and $\langle,\rangle_{s}$ is nondegenerate on ${\lie a}$, then
the metric $\langle,\rangle_{\g}-\frac1{\lambda}\langle,\rangle_{s}$ on $\tilde \g=\g\rtimes\lie a$ is Einstein with $\widetilde\Ric=\lambda \id$.
\item
\label{item:nonpseudoIwasawalambdazero}
If $\lambda=0$, $D\neq0$, $D=H^s$ for some  $H\in\lie a$ and $\langle,\rangle_{s}$ is zero, then for any metric $\langle,\rangle_{\lie a}$
on $\lie a$ such that $\langle H,X\rangle_{\lie a}=\Tr X$ for all $X$, the metric $\langle,\rangle_\g+\langle,\rangle_{\lie a}$ is Ricci-flat on the Lie algebra $\tilde \g=\g\rtimes\lie a$.
\item
\label{item:nonpseudoIwasawaD0H0}
If $\lambda=0$, $D=0$, elements of $\lie a$ are trace-free and $\langle,\rangle_{s}$ is zero, then for any metric $\langle,\rangle_{\lie a}$ on $\lie a$ the metric $\langle,\rangle_\g+\langle,\rangle_{\lie a}$ is Ricci-flat on the Lie algebra $\tilde \g=\g\rtimes\lie a$.
\item
\label{item:nonpseudoIwasawaD0Hnonzero}
If $\lambda=0$, $D=0$, elements of $\lie a$ are not all trace-free and $\langle,\rangle_{s}$ is zero, then for any metric $\langle,\rangle_{\lie a}$ on $\lie a$ which restricts to a nondegenerate metric on the subspace $\lie a_0\subset\lie a$ of trace-free elements, the metric $\langle,\rangle_\g+\langle,\rangle_{\lie a}+\langle,\rangle_H$ is Ricci-flat on the Lie algebra $\tilde \g=\g\rtimes\lie (\lie a\oplus\Span{H})$, where $\adtilde H=0$ and $\langle H,X\rangle_H=\Tr X$.
\end{enumerate}
\end{proposition}
\begin{proof}
In case~\ref{item:nonpseudoIwasawalambdanonzero}, the condition that $\langle,\rangle_s$ is nondegenerate implies that the projection $\lie a\to \lie a^s$ is injective. By Proposition~\ref{prop:pseudoAzencottWilson},
the solvmanifold defined in the statement is isometric to a solvmanifold $\g\rtimes\lie a^s$ with a pseudo-Iwasawa decomposition, which is Einstein by Theorem~\ref{thm:constructionstandardextension}.

In case~\ref{item:nonpseudoIwasawalambdazero} we cannot assume that the projection
$\lie a\to \lie a^s$ is injective, so we proceed by a direct application of Proposition~\ref{prop:ricciofstandard}. Since $H$ satisfies~\eqref{eqn:definitionH}, it acts on $\g$ by~\eqref{eqn:derivationphi}, i.e. $H=\sum_r \epsilon_r(\phi_r)\Tr\phi_r$. As elements of $\lie a$ are normal, for $v,w\in\lie g$ we have
\[\widetilde\ric(v,w)=\bigl\langle D(v)-\sum_r \epsilon_r(\phi_r)^s(v)\Tr\phi_r,w\bigr\rangle=\langle D(v)-(\ad H)^s(v),w\rangle=0.\]
Arguing as in Lemma~\ref{lem:RicciPseudo-Iwasawa} we see that $\langle\ad v, D\rangle_{\Tr}=0$ for every derivation $D$, so
\[\langle \ad v,X\rangle = \langle \ad v, X^s\rangle_{\Tr{}} -\langle \ad v, X^a\rangle_{\Tr{}}=0,\quad X\in\lie a.\]
This implies $\widetilde\ric(v,X)=0$.

Finally, for $X,Y\in\lie a$, we have
\[
\langle X,Y\rangle_{\Tr{}}=\langle X^s,Y^s\rangle_{\Tr{}}+\langle X^a,Y^a\rangle_{\Tr{}}=\langle X^a,Y^a\rangle_{\Tr{}}=\langle X,Y^a\rangle_{\Tr{}}=-\langle X,Y\rangle,\]
where we have used Lemma~\ref{lemma:ScalarOfSelfAdjoint}; this implies that $\widetilde\ric(X,Y)=0$.

Cases~\ref{item:nonpseudoIwasawaD0H0} and \ref{item:nonpseudoIwasawaD0Hnonzero} are proved in the same way.
\end{proof}

%

\begin{thebibliography}{10}

\bibitem{Alekseevski:Classification}
D.~V. Alekseevski\u{\i}.
\newblock Classification of quaternionic spaces with transitive solvable group
  of motions.
\newblock {\em Izv. Akad. Nauk SSSR Ser. Mat.}, 39(2):315--362, 472, 1975.

\bibitem{AzencottWilson2}
R.~Azencott and E.~N. Wilson.
\newblock Homogeneous manifolds with negative curvature. {II}.
\newblock {\em Mem. Amer. Math. Soc.}, 8(178):iii+102, 1976.
\newblock \href {https://doi.org/10.1090/memo/0178}
  {\path{doi:10.1090/memo/0178}}.

\bibitem{BatatOnda:AlgebraicRicciSoliton}
W.~Batat and K.~Onda.
\newblock {A}lgebraic {R}icci {S}olitons of three-dimensional {L}orentzian
  {L}ie groups.
\newblock {\em J.Geom. Phys.}, 114:138–152, 2017.

\bibitem{BohmLafuente:AlekseevskyConjecture}
C.~B\"{o}hm and R.~A. Lafuente.
\newblock {N}on-compact {E}instein manifolds with symmetry.
\newblock arXiv:2107.04210 [math.DG].

\bibitem{Bourbaki:LieGropuCh123}
N.~Bourbaki.
\newblock {\em Lie groups and {L}ie algebras. {C}hapters 1--3}.
\newblock Elements of Mathematics (Berlin). Springer-Verlag, Berlin, 1989.
\newblock Translated from the French, Reprint of the 1975 edition.

\bibitem{Calvaruso:OnSemidirect}
G.~Calvaruso.
\newblock On semi-direct extensions of the {H}eisenberg group.
\newblock {\em Collect. Math.}, 72(1):1--23, 2021.
\newblock \href {https://doi.org/10.1007/s13348-019-00277-y}
  {\path{doi:10.1007/s13348-019-00277-y}}.

\bibitem{Chevalley}
C.~Chevalley.
\newblock Algebraic {L}ie algebras.
\newblock {\em Ann. of Math. (2)}, 48:91--100, 1947.
\newblock \href {https://doi.org/10.2307/1969217} {\path{doi:10.2307/1969217}}.

\bibitem{gleipnir}
D.~Conti.
\newblock {G}leipnir, a program to study the {N}ikolayevsky derivation on {L}ie
  algebras with a parameter.
\newblock \url{https://github.com/diego-conti/gleipnir}.

\bibitem{ContidelBarcoRossi:SigmaType}
D.~Conti, V.~del Barco, and F.~A. Rossi.
\newblock {D}iagram involutions and homogeneous {R}icci-flat metrics.
\newblock {\em Manuscripta Math.}, 2020.
\newblock \href {https://doi.org/10.1007/s00229-020-01225-y}
  {\path{doi:10.1007/s00229-020-01225-y}}.

\bibitem{ContiRossi:Construction}
D.~Conti and F.~A. Rossi.
\newblock {C}onstruction of nice nilpotent {L}ie groups.
\newblock {\em Journal of Algebra}, 525:311 -- 340, 2019.
\newblock \href {https://doi.org/10.1016/j.jalgebra.2019.01.020}
  {\path{doi:10.1016/j.jalgebra.2019.01.020}}.

\bibitem{ContiRossi:EinsteinNilpotent}
D.~Conti and F.~A. Rossi.
\newblock Einstein nilpotent {L}ie groups.
\newblock {\em J. Pure Appl. Algebra}, 223(3):976--997, 2019.
\newblock \href {https://doi.org/10.1016/j.jpaa.2018.05.010}
  {\path{doi:10.1016/j.jpaa.2018.05.010}}.

\bibitem{ContiRossi:EinsteinNice}
D.~Conti and F.~A. Rossi.
\newblock Indefinite {E}instein metrics on nice {L}ie groups.
\newblock {\em {F}orum {M}athematicum}, 32(6):1599--1619, 2020.
\newblock \href {https://doi.org/10.1515/forum-2020-0049}
  {\path{doi:10.1515/forum-2020-0049}}.

\bibitem{Dotti:RicciCurvature}
I.~Dotti~Miatello.
\newblock Ricci curvature of left invariant metrics on solvable unimodular
  {L}ie groups.
\newblock {\em Math. Z.}, 180(2):257--263, 1982.
\newblock \href {https://doi.org/10.1007/BF01318909}
  {\path{doi:10.1007/BF01318909}}.

\bibitem{EberleinHeber}
P.~Eberlein and J.~Heber.
\newblock Quarter pinched homogeneous spaces of negative curvature.
\newblock {\em Internat. J. Math.}, 7(4):441--500, 1996.
\newblock \href {https://doi.org/10.1142/S0129167X96000268}
  {\path{doi:10.1142/S0129167X96000268}}.

\bibitem{FernandezCulma}
E.~A. Fern\'{a}ndez-Culma.
\newblock Classification of nilsoliton metrics in dimension seven.
\newblock {\em J. Geom. Phys.}, 86:164--179, 2014.
\newblock \href {https://doi.org/10.1016/j.geomphys.2014.07.032}
  {\path{doi:10.1016/j.geomphys.2014.07.032}}.

\bibitem{Gong}
M.-P. Gong.
\newblock {\em Classification of nilpotent {L}ie algebras of dimension 7 (over
  algebraically closed fields and {R})}.
\newblock ProQuest LLC, Ann Arbor, MI, 1998.
\newblock Thesis (Ph.D.)--University of Waterloo (Canada).

\bibitem{Heber:noncompact}
J.~Heber.
\newblock Noncompact homogeneous {E}instein spaces.
\newblock {\em Invent. Math.}, 133(2):279--352, 1998.
\newblock \href {https://doi.org/10.1007/s002220050247} 
  {\path{doi:10.1007/s002220050247}}.

\bibitem{HELLELAND2020:WickRotations}
C.~Helleland.
\newblock {W}ick-rotations of pseudo-{R}iemannian {L}ie groups.
\newblock {\em Journal of Geometry and Physics}, 158:103902, 2020.
\newblock \href {https://doi.org/10.1016/j.geomphys.2020.103902}
  {\path{doi:10.1016/j.geomphys.2020.103902}}.

\bibitem{Jablonski:Survey}
M.~Jablonski.
\newblock {S}urvey: Homogeneous {E}instein manifolds.
\newblock arXiv:2111.09782 [math.DG].

\bibitem{Jablonski2}
M.~Jablonski.
\newblock Homogeneous {R}icci solitons are algebraic.
\newblock {\em Geom. Topol.}, 18(4):2477--2486, 2014.
\newblock \href {https://doi.org/10.2140/gt.2014.18.2477}
  {\path{doi:10.2140/gt.2014.18.2477}}.

\bibitem{Jablonski1}
M.~Jablonski.
\newblock Homogeneous {R}icci solitons.
\newblock {\em J. Reine Angew. Math.}, 699:159--182, 2015.
\newblock \href {https://doi.org/10.1515/crelle-2013-0044}
  {\path{doi:10.1515/crelle-2013-0044}}.

\bibitem{Knapp:LieBook}
A.~W. Knapp.
\newblock {\em Lie groups beyond an introduction}, volume 140 of {\em Progress
  in Mathematics}.
\newblock Birkh\"{a}user Boston, Inc., Boston, MA, 1996.
\newblock \href {https://doi.org/10.1007/978-1-4757-2453-0}
  {\path{doi:10.1007/978-1-4757-2453-0}}.

\bibitem{KondoTamaru:LorentzianNilpotent}
Y.~Kondo and H.~Tamaru.
\newblock {A} classification of left-invariant {L}orentzian metrics on some
  nilpotent lie groups.
\newblock arXiv:2011.09118 [math.DG].

\bibitem{Lauret:RicciSoliton}
J.~Lauret.
\newblock Ricci soliton homogeneous nilmanifolds.
\newblock {\em Math. Ann.}, 319(4):715--733, 2001.
\newblock \href {https://doi.org/10.1007/PL00004456}
  {\path{doi:10.1007/PL00004456}}.

\bibitem{Lauret:Finding}
J.~Lauret.
\newblock Finding {E}instein solvmanifolds by a variational method.
\newblock {\em Math. Z.}, 241(1):83--99, 2002.
\newblock \href {https://doi.org/10.1007/s002090100407}
  {\path{doi:10.1007/s002090100407}}.

\bibitem{Lauret:Einstein_solvmanifolds}
J.~Lauret.
\newblock Einstein solvmanifolds are standard.
\newblock {\em Ann. of Math. (2)}, 172(3):1859--1877, 2010.
\newblock \href {https://doi.org/10.4007/annals.2010.172.1859}
  {\path{doi:10.4007/annals.2010.172.1859}}.

\bibitem{Magnin}
L.~Magnin.
\newblock Determination of 7-dimensional indecomposable nilpotent complex {L}ie
  algebras by adjoining a derivation to 6-dimensional {L}ie algebras.
\newblock {\em Algebr. Represent. Theory}, 13(6):723--753, 2010.
\newblock \href {https://doi.org/10.1007/s10468-009-9172-3}
  {\path{doi:10.1007/s10468-009-9172-3}}.

\bibitem{MedinaRevoy}
A.~Medina and P.~Revoy.
\newblock {Alg\`{e}bres de Lie et produit scalaire invariant.}
\newblock {\em {Ann. Sci. \'{E}cole Norm. Sup. (4)}}, 18:553--561, 1985.

\bibitem{Nikolayevsky}
Y.~Nikolayevsky.
\newblock Einstein solvmanifolds and the pre-{E}instein derivation.
\newblock {\em Trans. Amer. Math. Soc.}, 363(8):3935--3958, 2011.
\newblock URL: \url{https://doi.org/10.1090/S0002-9947-2011-05045-2}.

\bibitem{Onda:ExampleAgebraicRicciSolitons}
K.~Onda.
\newblock {E}xample of algebraic {R}icci solitons in the pseudo-{R}iemannian
  case.
\newblock {\em Acta Math. Hungar.}, 144 (1):247–265, 2014.

\bibitem{Payne:TheExistence}
T.~L. Payne.
\newblock The existence of soliton metrics for nilpotent {L}ie groups.
\newblock {\em Geom. Dedicata}, 145:71--88, 2010.
\newblock \href {https://doi.org/10.1007/s10711-009-9404-z}
  {\path{doi:10.1007/s10711-009-9404-z}}.

\bibitem{Salamon:ComplexStructures}
S.~M. Salamon.
\newblock Complex structures on nilpotent {L}ie algebras.
\newblock {\em J. Pure Appl. Algebra}, 157(2-3):311--333, 2001.
\newblock \href {https://doi.org/10.1016/S0022-4049(00)00033-5}
  {\path{doi:10.1016/S0022-4049(00)00033-5}}.

\bibitem{Wears}
T.~H. Wears.
\newblock On {L}orentzian {R}icci solitons on nilpotent {L}ie groups.
\newblock {\em Math. Nachr.}, 290(8-9):1381--1405, 2017.
\newblock \href {https://doi.org/10.1002/mana.201500039}
  {\path{doi:10.1002/mana.201500039}}.

\bibitem{Will:RankOne}
C.~Will.
\newblock Rank-one {E}instein solvmanifolds of dimension 7.
\newblock {\em Differential Geom. Appl.}, 19(3):307--318, 2003.
\newblock \href {https://doi.org/10.1016/S0926-2245(03)00037-8}
  {\path{doi:10.1016/S0926-2245(03)00037-8}}.

\bibitem{Wolter:EinsteinSolvable}
T.~H. Wolter.
\newblock Einstein metrics on solvable groups.
\newblock {\em Math. Z.}, 206(3):457--471, 1991.
\newblock \href {https://doi.org/10.1007/BF02571355}
  {\path{doi:10.1007/BF02571355}}.

\bibitem{Yan:Pseudo-RiemannianEinsteinhomogeneous}
Z.~Yan.
\newblock {P}seudo-{R}iemannian {E}instein metrics on noncompact homogeneous
  spaces.
\newblock {\em J.Geom.}, 111, 2020.
\newblock \href {https://doi.org/10.1007/s00229-020-01249-4}
  {\path{doi:10.1007/s00229-020-01249-4}}.

\bibitem{Yan:pseudoalgebraic}
Z.~Yan.
\newblock Pseudo-algebraic {R}icci solitons on {E}instein nilradicals.
\newblock {\em Advances in Geometry}, 2021.

\bibitem{YanDeng:DoubleExtensionRiemNilsolitons}
Z.~Yan and S.~Deng.
\newblock Double extensions on riemannian ricci nilsolitons.
\newblock {\em J Geom Anal}, 2021.
\newblock \href {https://doi.org/10.1007/s12220-021-00636-x}
  {\path{doi:10.1007/s12220-021-00636-x}}.

\end{thebibliography}

\small\noindent Dipartimento di Matematica e Applicazioni, Universit\`a di Milano Bicocca, via Cozzi 55, 20125 Milano, Italy.\\
\texttt{diego.conti@unimib.it}\\
\texttt{federico.rossi@unimib.it}

\end{document}